\documentclass[12pt]{amsart}

\usepackage{titletoc}
\usepackage{diagbox}
\usepackage{imakeidx}
\makeindex
\usepackage{pdfpages}
\usepackage{tikz}
\usepackage[colorlinks=true, linkcolor=blue, anchorcolor=blue, citecolor=blue, filecolor=blue, menucolor= blue, urlcolor=blue,pdfencoding=auto, psdextra]{hyperref}
\usepackage{bookmark}
\usepackage{amsmath,amssymb,amsthm,amscd}
\usepackage{verbatim}
\usepackage{comment}
\usepackage{multirow}
\usepackage{mathtools}
\usepackage{enumitem} 
\usepackage{pgf,tikz,mathrsfs}
\usetikzlibrary{arrows}
\usepackage[normalem]{ulem}
\usepackage{comment}
\usepackage{multicol}
\usepackage{xfrac,xcolor}
\usepackage[margin=1in]{geometry}
\usepackage{graphicx}
\numberwithin{equation}{section}
\usepackage{pstricks-add,pst-plot,pst-func}
\usepackage{cleveref}
\usepackage{caption}
\usepackage{subcaption}

\newtheorem{theorem}{Theorem}[section]
\newtheorem{proposition}[theorem]{Proposition}
\newtheorem{lemma}[theorem]{Lemma}
\newtheorem{corollary}[theorem]{Corollary}

\newtheorem{prob*}{Problem}
\newtheorem{question}[theorem]{Question}
\newtheorem{questions}[theorem]{Questions}

\newtheorem*{theorem*}{Theorem}

\theoremstyle{definition}

\newtheorem{notation}[theorem]{Notation}
\newtheorem{example}[theorem]{Example}
\newtheorem{remark}[theorem]{Remark}

\newtheorem{thmx}{Theorem}

\newtheorem{conjx}[thmx]{Conjecture}

\numberwithin{equation}{section}

\definecolor{green}{rgb}{.1,.75,.1}

\usepackage{xparse}
\usepackage{tikz}
\usetikzlibrary{matrix,backgrounds}
\pgfdeclarelayer{myback}
\pgfsetlayers{myback,background,main}

\tikzset{mycolor/.style = {line width=1bp,color=#1}}%
\tikzset{myfillcolor/.style = {draw,fill=#1}}%

\NewDocumentCommand{\highlight}{O{blue!40} m m}{%
	\draw[mycolor=#1] (#2.north west)rectangle (#3.south east);
}

\NewDocumentCommand{\fhighlight}{O{blue!40} m m}{%
	\draw[myfillcolor=#1] (#2.north west)rectangle (#3.south east);
}

\newcommand{\ZZ}{ \ensuremath{\mathbb{Z}}}
\newcommand{\CC}{ \ensuremath{\mathbb{C}}}

\newcommand{\PP}{ \ensuremath{\mathbb{P}}}

\newcommand{\reg}{\ensuremath{\mathrm{reg}}\hspace{1pt}}

\newcommand{\Tor}{\ensuremath{\mathrm{Tor}}\hspace{1pt}}

\newcommand{\comp}[2]{$\mathbf{#1}\quad{#2}$}

\definecolor{MyDarkGreen}{cmyk}{0.7,0,1,0}

\def\cocoa{{\hbox{\rm C\kern-.13em o\kern-.07em C\kern-.13em o\kern-.15em A}}}

\usepackage{etoolbox}
\patchcmd{\subsection}{-.5em}{.5em}{}{}
\patchcmd{\subsection}{2}{3}{}{}
\makeatletter
\newsavebox\myboxA
\newsavebox\myboxB
\newlength\mylenA

\newcommand*\xoverline[2][0.75]{%
    \sbox{\myboxA}{$\m@th#2$}%
    \setbox\myboxB\null
    \ht\myboxB=\ht\myboxA%
    \dp\myboxB=\dp\myboxA%
    \wd\myboxB=#1\wd\myboxA
    \sbox\myboxB{$\m@th\overline{\copy\myboxB}$}
    \setlength\mylenA{\the\wd\myboxA}
    \addtolength\mylenA{-\the\wd\myboxB}%
    \ifdim\wd\myboxB<\wd\myboxA%
       \rlap{\hskip 0.5\mylenA\usebox\myboxB}{\usebox\myboxA}%
    \else
        \hskip -0.5\mylenA\rlap{\usebox\myboxA}{\hskip 0.5\mylenA\usebox\myboxB}%
    \fi}
\makeatother

\begin{document}

\title{Intersection of curves in $\PP^4$}
\date{January 19, 2026}

\author[\tiny Chiantini]{Luca Chiantini}
\address[L.~Chiantini]{Dipartmento di Ingegneria dell'Informazione e Scienze Matematiche, Universit\`a di Siena, Italy}
\email{luca.chiantini@unisi.it}

\author[Farnik]{{\L}ucja Farnik}
\address[{\L}.~Farnik]{Department of Mathematics, University of the National Education Commission, Krakow,
   Podcho\-r\c a\.zych~2,
   PL-30-084 Krak\'ow, Poland}
\email{lucja.farnik@gmail.com}

\author[Favacchio]{Giuseppe Favacchio}
\address[G.~Favacchio]{Dipartimento di Ingegneria, Universit\`a degli studi di Palermo,
Viale delle Scienze,  90128 Palermo, Italy}
\email{giuseppe.favacchio@unipa.it}

\author[Harbourne]{Brian Harbourne}
\address[B.~Harbourne]{Department of Mathematics,
University of Nebraska,
Lincoln, NE 68588-0130 USA}
\email{brianharbourne@unl.edu}

\author[Migliore]{Juan Migliore} 
\address[J.~Migliore]{Department of Mathematics,
University of Notre Dame,
Notre Dame, IN 46556 USA}
\email{migliore.1@nd.edu}

\author[Szemberg]{Tomasz Szemberg}
\address[T.~Szemberg]{Department of Mathematics, University of the National Education Commission, Krakow,
   Podcho\-r\c a\.zych~2,
   PL-30-084 Krak\'ow, Poland}
\email{tomasz.szemberg@gmail.com}

\author[Szpond]{Justyna Szpond}
\address[J.~Szpond]{Department of Mathematics, University of the National Education Commission, Krakow,
   Podcho\-r\c a\.zych~2,
   PL-30-084 Krak\'ow, Poland}
\email{szpond@gmail.com}

\begin{abstract}
Given positive integers $n$, $d_1$ and $d_2$, it is natural to seek a sharp upper bound for the number of intersection points of reduced, irreducible curves $C_1$ and $C_2$ of those degrees in $\PP^n$. When $n=2$ this is trivial. When $n=3$ this problem was solved independently by Diaz and by Giuffrida. They showed that two curves achieving the maximum number of intersection points have to be rational curves on a smooth surface of minimal degree, i.e., a quadric surface. As such, these curves are far from being arithmetically Cohen-Macaulay (ACM). It is noteworthy that Hartshorne and Mir\'o-Roig also addressed this problem for curves in $\PP^3$ under the assumption that the curves {\it are} ACM; this latter paper produced very different results from those of Diaz and Giuffrida, and introduced very deep techniques. Diaz and Giuffrida also gave initial results when $n > 3$. In this note we continue this study when $n=4$. We introduce an invariant, $B(d_1,d_2)$.  We prove that when both curves lie on a surface of minimal degree (now a cubic surface) then $|C_1 \cap C_2| \leq B(d_1,d_2)$, and for any $d_1,d_2$ sharp examples exist (although only one of our two curves is rational). We conjecture that $B(d_1,d_2)$ is always an upper bound, among arbitrarily chosen irreducible curves of the given degrees. We prove this conjecture in many cases; one of these is when at least one of the two curves is ACM, giving a link to the work of Hartshorne and Mir\'o-Roig. Our approach focuses on the genera of $C_1$, $C_2$ and especially the genus, $g$, of $C_1 \cup C_2$. We introduce a second invariant, $B_g(d_1,d_2)$, which is always an upper bound; however, it may happen that $B_g(d_1,d_2) > B(d_1,d_2)$ so in these cases this approach does not lead immediately to a proof of the conjecture for the degrees $d_1,d_2$ in question. We give some situations where we can get around this problem, one being the ACM situation just mentioned. We also show, when $d_1, d_2$ are both odd, that smooth, irreducible, rational  curves achieving $|C_1 \cap C_2| = B(d_1,d_2)$ exist {\it not} lying on a cubic surface, giving a stark contrast to the situation in $\PP^3$. 
Our main approach is to study the $h$-vector and the geometry of the general hyperplane section $\Gamma$ of $C_1 \cup C_2$. Finally, using a lifting result of Huneke and Ulrich, we make conclusions about what surfaces may contain $C_1 \cup C_2$.
\end{abstract}

\maketitle

\setcounter{tocdepth}{1}
\tableofcontents

\section{Introduction}

In this work, a \emph{curve} refers to a reduced, equidimensional variety of dimension~$1$. Throughout the paper, we assume that $C_1$ and $C_2$ are reduced, irreducible, and nondegenerate curves of degrees $d_1$ and $d_2$, respectively. A curve is {\it arithmetically Cohen-Macaulay (ACM)} if its coordinate ring is Cohen-Macaulay, and several cohomological criteria exist to determine this. We are interested in finding an upper bound for the number, $|C_1 \cap C_2|$  (counted without multiplicity), of intersection points of curves $C_1,C_2$ of degrees $d_1,d_2$ in $\PP^n$ over the complex numbers. The following natural questions  have been addressed by various authors (as we will see).
 
\begin{questions} \label{questions}
\ 
    \begin{enumerate}
    
    \item \label{Q1}  What kinds of upper bounds can be given  for $|C_1\cap C_2|$ and how do they depend on the dimension, $n$, of the ambient projective space?
    
        \item \label{Q2} Given $n$, is there a sharp bound that depends only on $d_1$ and $d_2$?
        
        \item \label{Q3} If $C_1$ and $C_2$ attain the maximal number of intersection points, what additional geometric or algebraic constraints must they satisfy?
        
        \item \label{Q4} What roles do the arithmetic genera of the curves play?
    \end{enumerate}
\end{questions}

The following two additional questions were motivated by analogous questions posed by Hartshorne and Mir\'o-Roig \cite{hartshorne2015} for ACM curves in $\PP^3$.

{\it 
\begin{enumerate} \setcounter{enumi}{4}

     \item \label{Q5} If two curves have the maximum number of intersection points, does this force the two curves to lie in a common surface of minimal degree? 
    
    \item \label{Q6} If the maximum is attained, is the union of the two curves necessarily arithmetically Cohen-Macaulay? 
    
\end{enumerate}
}

\noindent Quoting Hartshorne and Mir\'o-Roig,
\begin{center}
    {\it ``There seems to be scant attention to these questions in the literature.'' }
\end{center}

\noindent  Our goal is to address all these questions in the case of curves in $\PP^4$. We begin here with a quick review of the known results. 
 
If the ambient space is $\mathbb{P}^2$, then of course the set $C_1 \cap C_2$ may contain up to $d_1d_2$ points, with the upper bound being achieved when the curves intersect transversally. The arithmetic genus of either curve does not play a relevant role here. 

Several papers have the aim of obtaining results in higher dimensional projective spaces \cite{diaz}, \cite{Giu86}, \cite{Abhyankar1990}, \cite{Abhyankar1991}, \cite{Miro-Roig2005}, \cite{hartshorne2015}.
 
Turning to the details, for curves $C_1,C_2$ in $\mathbb{P}^3$, Diaz \cite{diaz} and, independently, Giuffrida \cite{Giuffrida86} established the upper bound:
\begin{equation}\label{eq: DG bound in P3}
|C_1\cap C_2| \leq (d_1-1)(d_2-1)+1.
\end{equation}
Furthermore, they showed that this bound is achieved if and only if both curves lie on the same smooth quadric, in particular a surface of minimal degree in $\mathbb{P}^3$ and have arithmetic genus 0, with divisor types $(d_1-1,1)$ and $(1,d_2-1)$. 
Moving beyond this, Giuffrida \cite{Giu86}  showed that the maximum intersection for curves on a cubic surface in $\PP^3$ but not on a quadric is on the order of $\frac{5}{9} d_1d_2$.

The approach of Abhyankar, Chandrasekar and Chandru \cite{Abhyankar1990, Abhyankar1991} was different from the approach of the other papers on this topic, in that they focused on curves not contained in hypersurfaces of unexpectedly low degree, and they found an application to computer vision problems. On the other hand, as will be seen in a moment, curves on hypersurfaces of unexpectedly low degree are important because this is the setting where sharp examples arise. 

Continuing in $\PP^3$, the paper \cite{Miro-Roig2005} of Mir\'o-Roig and Ranestad posed an interesting variant of this problem,  bounding the number of intersection points of two {\it ACM} curves under the assumption that for both curves, the defining ideals  are given by the maximal minors of  homogeneous matrices {\it with linear entries}. 
Then the paper \cite{hartshorne2015} by Hartshorne and Mir\'o-Roig  extended this,  looking for the maximum number of intersection points of two reduced, irreducible ACM curves in $\PP^3$ with no restrictions on the defining matrix. They heavily  used results from liaison theory, and the approach of their paper was through a 
careful study of the $h$-vectors of $C_1, C_2$ and $C_1 \cup C_2$. This is a classical approach and is also at the heart of this paper. One of our main results is in the context of ACM curves in $\PP^4$ (see Theorem~\ref{theorem d}).

The situation for curves in higher-dimensional projective spaces $\mathbb{P}^n$ is much less clear. Giuffrida established the following bound in \cite[Theorem]{giuffrida}:
\begin{equation}\label{eq: G bound in Pn}
|C_1\cap C_2|  \leq (d_1-n+2)(d_2-n+2)+n-2,
\end{equation}
which, for $n=3$, coincides exactly with \eqref{eq: DG bound in P3}.
Moreover, Giuffrida demonstrated in \cite[Example]{giuffrida} that the bound in \eqref{eq: G bound in Pn} is sharp and is attained when $C_1$ and $C_2$ are rational normal curves (of degree $n$) meeting in $n+2$ points.

Diaz also addressed curves in $\mathbb{P}^n$ \cite[(13) Corollary]{diaz}, giving the bound
\[
| C_1 \cap C_2 | \leq (d_1 -n +1)d_2
\]
under the assumption that $C_1$ (only) is irreducible and nondegenerate, no component of $C_2$ is a line or is equal to $C_1$, and $d_1 \geq n+2$.

Thus, in the case of central interest in this work -- namely, that of curves in $\PP^4$ -- the number of intersection points of two curves $C_1$ and $C_2$ is bounded above by the function $B_{DG}(d_1,d_2)$, which we refer to as the \emph{Diaz-Giuffrida bound}: 
\begin{equation}\label{eq: Diaz G}
B_{DG}(d_1,d_2)=
\left\{
\begin{array}{ll}
   (d_1-2)(d_2-2)+2  & \mbox{if }  d_1\leq 5 \mbox{ or } d_2\leq 5,\\
   & \\
   \min\{(d_1-2)(d_2-2)+2, (d_1-3)d_2, d_1(d_2-3)\} & \mbox{otherwise.}
\end{array}
\right.
\end{equation}
Our goal is to improve this bound. To this end, we introduce the following two functions:
\begin{equation}\label{eq: our bound}
B(d_1,d_2)=\left\{
\begin{array}{ccl}
     \dfrac{d_1(d_2-1)}{2} & \rm{ if } &  \mbox{\textbf{both even} and } d_1\leq d_2\\
     &&\mbox{or } d_1 \mbox{ \textbf{even}, } d_2\mbox{ \textbf{odd}},\\[10pt]
     \dfrac{(d_1-1)d_2}{2} & \rm{ if } & \mbox{\textbf{both even} and } d_1\geq d_2\\
     &&\mbox{or } d_1 \mbox{ \textbf{odd}, } d_2\mbox{ \textbf{even}},\\[10pt]
    \dfrac{(d_1-1)(d_2-1)}{2}+1 & \rm{ if } & \mbox{\textbf{both odd}}.
\end{array}
\right.
\end{equation}
Furthermore, setting $d = d_1 + d_2 = 4k + \ell$ for $\ell \in \{0,1,2,3\}$, we define:
\begin{equation}\label{eq: genus bound}
B_g(d_1,d_2)=\left\{
\begin{array}{ccc}
    \dfrac{d^2-4d}{8}+2 && \ell=0, \\[10pt]
    \dfrac{d^2-4d+3}{8}+1  && \ell =1 \mbox{ or } \ell=3,\\[10pt]
    \dfrac{d^2-4d+4}{8}+1 && \ell=2.
\end{array}
\right.    
\end{equation}
At this point, the motivation behind the definitions of $B$ and $B_g$ may not be immediately apparent. However, their relevance will become clearer in the course of the paper 
(see Notation 
\ref{not: genus bound Bg} and  \Cref{{cor: genus bound explicit}} for how \eqref{eq: genus bound} is derived). 
For now, we note that $B(d_1, d_2)$ is  asymptotically half the Diaz-Giuffrida bound.

Mimicking the situation in $\PP^3$
we  would first like to know the bound for $|C_1 \cap C_2|$ on an irreducible surface of minimal degree (a cubic surface), and understand sharpness of our bound on this surface. Our first result deals with this situation.

\vspace{.1in}

\begin{thmx} \label{theorem a}
\textit{ Let $S$ be an irreducible, nondegenerate cubic surface in $\mathbb P^4$. Then any two reduced, irreducible curves $C_1, C_2$ in $S$ of degrees $d_1, d_2$ satisfy $|C_1 \cap C_2| \leq B(d_1,d_2)$. Furthermore, there exist in $S$ two reduced, irreducible curves $C_1$ and $C_2$ of degrees $d_1$ and $d_2$ such that $|C_1\cap C_2|=B(d_1,d_2)$, for any $d_1,d_2$.  For each choice of $d_1,d_2$ we construct  an example of  curves achieving this bound in terms of their divisor classes  on $S$. }
\end{thmx}

\vspace{.1in}

For the proof, see Theorem \ref{bd on cubic} and Remark \ref{sing cubic}. 

The main goal of this paper is to study to what extent $B(d_1,d_2)$ is an upper bound for $|C_1 \cap C_2|$ among all choices of reduced, irreducible, nondegenerate curves of degrees $d_1, d_2$ in $\PP^4$. We pose the following conjecture: 

\vspace{.1in}

\begin{conjx} \label{conjecture b}
{\it Fix positive integers $d_1,d_2$ and let $C_1,C_2$ be reduced, irreducible, nondegenerate curves of degrees $d_1,d_2$ in $\PP^4$. Then \begin{enumerate}
    \item[1)] the number of intersection points of $C_1$ and $C_2$ is at most $B(d_1,d_2)$;
    
    \item[2)] equality can only be achieved in one of the following situations:

    \begin{itemize}

\item both curves lie on a common cubic surface (and then for arbitrary $d_1,d_2$);

\item $d_1$ and $d_2$ are odd, both curves have arithmetic genus 0, and they lie on a common complete intersection quartic surface;

\item $d_1 = d_2 = 4$ and the curves are linked in a complete intersection of type $(2,2,2)$.
        
    \end{itemize}
    
\end{enumerate}}
\end{conjx}

\vspace{.1in}

Our main tools in this paper will require us to assume that $d_1, d_2 \geq 6$ and neither curve lies on a cubic surface, so it is of interest to study the case of low degree and the case where one curve does lie on a cubic.  Of course if $d_i \leq 3$ then $C_i$ is degenerate, so we can assume $d_1,d_2 \geq 4$. We show (\Cref{quintic}) that when $d \leq 5$, an irreducible curve of degree $d$ lies on an irreducible cubic surface.

\begin{thmx}
{\it 
Let $S$ be an irreducible cubic surface in $\PP^4$ and let $C_1\subset S$ and $C_2\not\subset S$ be reduced, irreducible, nondegenerate curves in $\PP^4$. Then
\begin{enumerate}
    \item[1)] $ |C_1 \cap C_2 |  \leq  \ |S \cap C_2|\leq 2d_2-1$ (Lemma \ref{C int S});
    \item[2)] $|C_1\cap C_2|\leq B(d_1,d_2)$ with  equality 
    occurring only if
    \begin{itemize}[left=.7em, label=-]
    \item[] either  $d_1=d_2=4$ and $C_1,C_2$ are linked in a complete intersection of type $(2,2,2)$ (\Cref{low deg th 5.9}); 
    \item[] or $d_1=5$, $d_2$ is odd and both curves have arithmetic genus $0$ (\Cref{bd for quintic}).
    \end{itemize}
\end{enumerate}
} 
\end{thmx}
\noindent In view of the above, it is not restrictive to assume $d_1, d_2 \geq 6$. For the rest of this overview we will make this assumption.

The following theorem collects our main results toward a resolution of Conjecture \ref{conjecture b}.

\begin{thmx} \label{theorem d} {\it Assume that $C_1, C_2$ are reduced, irreducible and nondegenerate curves of degrees $\geq 6$ and that they do not lie on a common cubic surface. Then 
\begin{itemize}
    \item[1)] If at least one of the curves is ACM (\Cref{acm curves}) then \Cref{conjecture b} holds;
    
    \item[2)] If $d_1$ and $d_2$ are both odd then smooth rational curves $C_1,C_2$ exist with $|C_1 \cap C_2| = B(d_1,d_2)$ (\Cref{lem: curves on del Pezzo quartic});
    
    \item[3)] If $(d_2-d_1)^2\leq M(d_1,d_2)$, where $M(d_1,d_2)$ is as specified in Table \ref{conj cases} (\Cref{B minus Bg}), we have $|C_1 \cap C_2| \leq B(d_1,d_2)$.  If the first inequality is strict then so is the second.
   
\end{itemize}
    In particular, combining 2) and 3), \Cref{conjecture b} holds whenever $d_1 = d_2$.}
\end{thmx}

\vspace{.1in}

For curves not on a cubic surface, the known results for $\PP^3$ would lead one to expect that $B(d_1,d_2)$ is not only a bound, but in fact a strict upper bound for the number of intersection points of curves of degrees $d_1$ and $d_2$ in $\PP^4$. Thus part 2) of the above theorem  was a surprise to us.

Comparing our bound in Theorem \ref{theorem a} with the bound of Diaz and Giuffrida for curves in $\PP^4$ (see \eqref{eq: Diaz G}), we note that our result provides an asymptotically sharper estimate improving the bound by a factor of $1/2$ in this context.
In Table \ref{tab: compare G and POLITUS}, we illustrate how our upper bound (\textbf{in boldface}) compares to that of Diaz and Giuffrida for some small degrees. 

\renewcommand{\arraystretch}{1.3}

\begin{table}[h!]
    \centering
    \begin{tabular}{c|c|c|c|c|c|c|c}
    \diagbox[]{$d_1$}{$d_2$} & 4 & 5 & 6 & 7 & 8 & 9 & 100 \\
    \hline
    4     & \comp{6}{6} & \comp{8}{8} & \comp{10}{10} & \comp{12}{12} & \comp{14}{14} & \comp{16}{16} & \comp{198}{198}\\
    5     & \comp{8}{8} & \comp{9}{11} & \comp{12}{14} & \comp{13}{17} & \comp{16}{20} & \comp{17}{23} & \comp{200}{296} \\
    6     & \comp{10}{10} & \comp{12}{14} & \comp{15}{18} & \comp{18}{21} & \comp{21}{24} & \comp{24}{27} & \comp{297}{300} \\
    7     & \comp{12}{12} & \comp{13}{17} & \comp{18}{21} & \comp{19}{27} & \comp{24}{32} & \comp{25}{36} & \comp{300}{400}\\
    8     & \comp{14}{14} & \comp{16}{20} & \comp{21}{24} & \comp{24}{32} & \comp{28}{38} & \comp{32}{44} & \comp{396}{500}\\
    9     & \comp{16}{16} & \comp{17}{23} & \comp{24}{27} & \comp{25}{36} & \comp{32}{44} & \comp{33}{51} & \comp{400}{600}\\
    100     & \comp{198}{198} & \comp{200}{296} & \comp{297}{300} & \comp{300}{400} & \comp{396}{500} & \comp{400}{600} & \comp{4950}{9700}\\
    \end{tabular}
    \bigskip
    \caption{Comparison between \eqref{eq: our bound} and \eqref{eq: Diaz G}.}
    \label{tab: compare G and POLITUS}
\end{table}

Our method for Theorem \ref{theorem a} (Theorem \ref{bd on cubic})  is by a direct examination of the curves arising as divisors on the one-point blowup of $\PP^2$ embedded in $\PP^4$ by the linear system of conics passing through that point.

 Our main approach involves deriving a bound, $B_g(d_1,d_2)$, based on the arithmetic genus of the union. Now we describe this approach.
Proposition \ref{rosa fact} recalls the well-known fact that
\[
|C_1 \cap C_2| \leq g - g_1 - g_2 + 1
\]
where  $g$ is the arithmetic genus of $C = C_1 \cup C_2$, and $g_i$ is the arithmetic genus of $C_i$, $i = 1,2$. This already shows that maximizing $|C_1\cap C_2|$ should be closely tied to maximizing $g$, and it suggests that having $g_1 = g_2 = 0$ would be advantageous to maximizing $|C_1\cap C_2|$ (if it is possible). 

To understand $g$ we invoke the Hilbert function (more precisely the $h$-vector) of the general hyperplane section, $\Gamma$, of $C$. Proposition \ref{genus formula} recalls the connection between $g$, the $h$-vector of $\Gamma$, and a certain submodule of the deficiency module of $C$. When $C$ is not ACM, one of our challenges is to ``lift" properties that we deduce about $\Gamma$ back up to $C$.
These considerations allow us to produce a specific bound $B_g(d_1,d_2)$ for the number of intersection points.  This {\it genus bound} and the above connections are discussed in Section \ref{sec: genus bd} and then starting with Section~\ref{sec: Bg leq B}.

When $B_g(d_1,d_2) \leq B(d_1,d_2)$, we are done: then it necessarily follows that $|C_1 \cap C_2| \leq B_g(d_1,d_2) \leq B(d_1,d_2)$. However, sometimes it is not true that $B_g(d_1,d_2) \leq B(d_1,d_2)$ (see Remark \ref{d1,d2 far apart}) and we discuss some of its limitations in Section \ref{sec: beyond Bg}.

In Figure \ref{fig:Bg vs G} we compare $B_g(d_1,d_2)$ with the minimum of the bounds given by Giuffrida and by Diaz.  It shows that ``most of the time" $B_g(d_1,d_2)$ gives an improvement over the previously known bounds. 
In Figure~\ref{fig:Bg vs B}  we compare $B_g (d_1,d_2)$ with $B(d_1,d_2)$.

\begin{figure}
     \centering
     \begin{subfigure}[t]{0.45\textwidth}
         \centering
         \includegraphics[width=\textwidth]{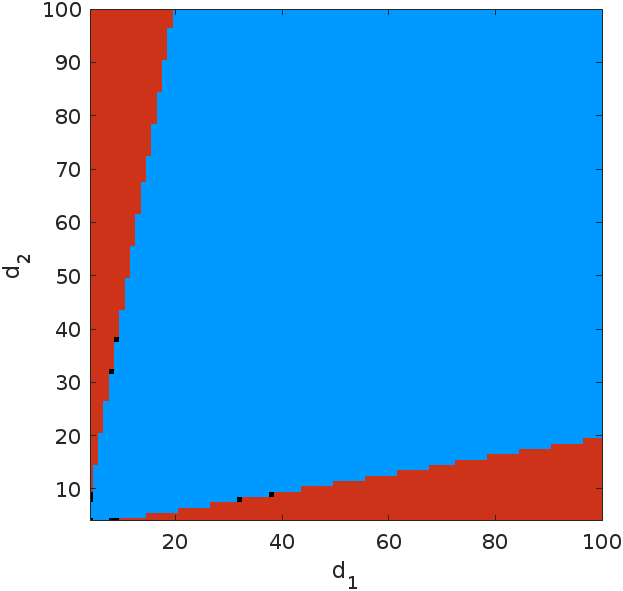}
         \caption{The blue pixels mean $B_g<B_{DG}$,  black means $B_g=B_{DG}$ and red means
         $B_g>B_{DG}$.}
     \end{subfigure}\hspace{1cm}
     \begin{subfigure}[t]{0.45\textwidth}
         \centering
         \includegraphics[width=\textwidth]{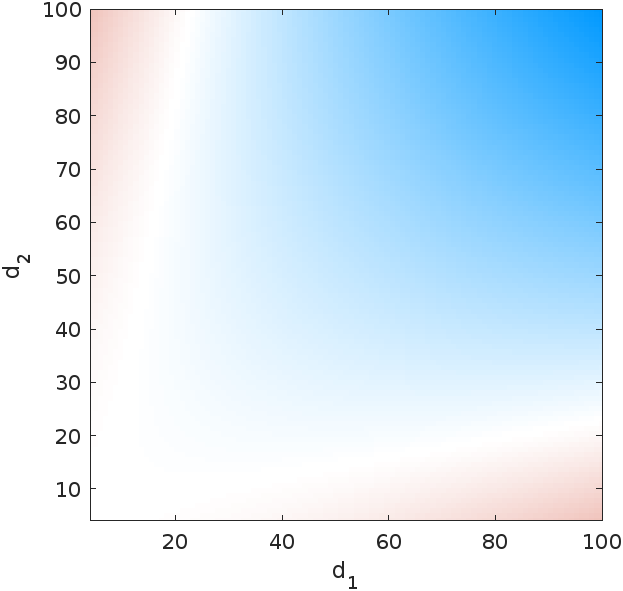}
         \caption{Each pixel's color encodes the difference between the value of 
$B_g$ and $B_{DG}$. Color intensity is scaled proportionally to the absolute value of the difference, normalized by the maximum across the domain.}
     \end{subfigure}
        \caption{Comparison between the genus bound $B_g(d_1,d_2)$ and the Giuffrida-Diaz bound $B_{DG}(d_1,d_2)$.}
        \label{fig:Bg vs G}
\end{figure}

\begin{figure}[h!]
     \centering
     \begin{subfigure}[t]{0.45\textwidth}
         \centering
         \includegraphics[width=\textwidth]{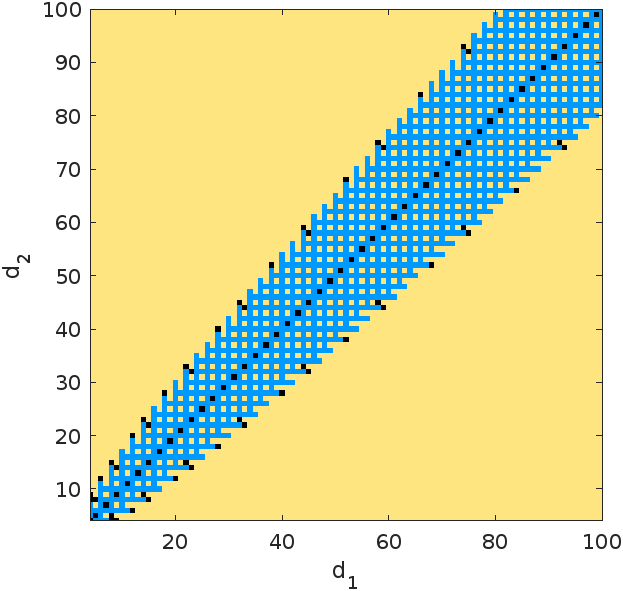}
         \caption{The blue pixels mean $B_g<B$,  black means $B_g = B$, and yellow means $B_g > B$.}
     \end{subfigure}\hspace{1cm}
     \begin{subfigure}[t]{0.45\textwidth}
         \centering
         \includegraphics[width=\textwidth]{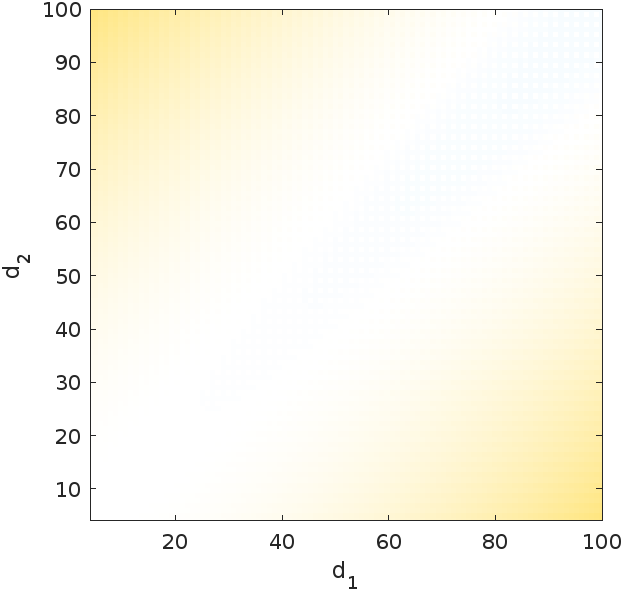}
         \caption{Each pixel's color encodes the difference between the value of 
$B_g$ and $B$. Color intensity is scaled proportionally to the absolute value of the difference, normalized by the maximum across the domain.}
     \end{subfigure}
        \caption{Comparison between $B_g(d_1,d_2)$ and $B(d_1,d_2)$.}
        \label{fig:Bg vs B}
\end{figure}

With regard to Question \ref{Q1} of Questions \ref{questions}, we can see a stark difference between the bounds for $\PP^3$ and $\PP^4$, and one of our open questions (see Section \ref{sec: open questions}) asks what happens in higher projective space. The above results give our contributions to Questions \ref{Q2} and \ref{Q5} (the latter having a negative answer). We also give a negative answer to Question \ref{Q6} in Remark \ref{union not ACM} (see also Proposition \ref{prop: B holds odd degrees}). Our contributions to Questions \ref{Q3} and \ref{Q4} really get into the mechanics of our methods, but we note the importance of Proposition \ref{rosa fact} that directly relates to Question \ref{Q4}. Note also that  one of the two extreme curves (in each case) described in Theorem \ref{bd on cubic} does not have arithmetic genus 0, unlike the conclusion for $\PP^3$ established by Diaz and Giuffrida.

\newpage

\section{The geography of this paper}

Throughout this paper we have many technical and seemingly unmotivated results. We begin our work by giving an overview of these technical results and where they will be important.

We have  the following facts:

\begin{enumerate}

\item Let $C$ be a curve and let $\Gamma$ be its general hyperplane section. We denote by $h_\Gamma$ the Hilbert function of $\Gamma$ in $H = \mathbb P^3$ and by $\underline{h}$ the corresponding $h$-vector (i.e., the first difference of $h_\Gamma$). They give equivalent information: either of $h_\Gamma$ or $\underline{h}$ can be deduced from the other.

In Proposition \ref{genus formula} we recall that the arithmetic genus $g$ of $C$ satisfies
\[
g = \sum_{i=1}^\ell [d - h_\Gamma(i)] - k
\]
where $d = \deg(C)$,  $k$ is the vector space dimension of a certain submodule $K$  of the Hartshorne-Rao module $M(C)$ of $C$, and $\ell \gg 0$. Moreover,  $C$ is ACM if and only if $k =0$.

We denote by $g(\underline{h})$ the integer
\[
g(\underline{h}) = \sum_{i=1}^\ell [d - h_\Gamma(i)],
\]
which is the arithmetic genus of $C$ if it is ACM.

\item \label{item2} Given irreducible curves $C_1$ and $C_2$ of degrees $d_1, d_2$ and arithmetic genera $g_1, g_2$, we consider the union $C = C_1 \cup C_2$. Let $\Gamma$ be the general hyperplane section of $C$ and let $\underline{h}$ be the $h$-vector of $\Gamma$. Using  Proposition \ref{rosa fact} and \ref{genus formula} and Remark \ref{formal calc} it follows that
\[
|C_1 \cap C_2| \leq  g - g_1 - g_2 +1 \leq g+1 \leq  
g(\underline{h}) +1.
\]
The equality $|C_1 \cap C_2| = g(\underline{h}) +1$ holds  if and only if $g_1 = g_2 = 0$ and $C_1 \cup C_2$ is ACM.

\item For the central problem of this paper, we are given only $d_1, d_2$ and consequently their sum $d$. To find an upper bound for $|C_1 \cap C_2|$ (before comparing it to the target $B(d_1,d_2)$) we need to find an upper bound for $g(\underline{h})$ as $\underline{h}$ varies over all possible $h$-vectors $\underline{h}$ for $\Gamma$. So an important component of our work is to restrict the possible $\underline{h}$ that we consider by eliminating many $\underline{h}$.

\item The problem of finding a sharp bound for $|C_1 \cap C_2|$ is completely solved if $C_1$ and $C_2$ lie on a common cubic surface (Theorem \ref{bd on cubic} and Remark \ref{sing cubic}).

\item Assume $C_1$ lies on a cubic surface $S$ and $C_2$ does not lie on $S$. Then $|C_1 \cap C_2| \leq 2d_2 -1$ (Corollary \ref{cor of C int S}). If $d_1 \geq 6$ then one checks that this forces $|C_1 \cap C_2| < B(d_1,d_2)$. If $d_1$ is 4 or 5 (the only other possibilities for irreducible, nondegenerate curves) then these are handled in Lemma \ref{rnc} and Lemma \ref{bd for quintic} (see also Remark \ref{low deg th 5.9}).

\item Combining the latter two items, we can assume without loss of generality that neither curve lies on a cubic surface. Since we show that any nondegenerate irreducible curve of degree 5 lies on a cubic surface (Lemma \ref{quintic}), we can also assume that $d_1,d_2 \geq 6$.

\item Even though we cannot conclude that $\Gamma$ has the Uniform Position Property (UPP) since $C$ is not irreducible, we can assume that the points are in Linear General Position (LGP) in the hyperplane $H = \mathbb P^3$, i.e., that no four of the points lie on a plane (Lemma \ref{lgp rmk}).

\item 
Assume that the $h$-vector of $\Gamma$ has the form
\[
(1,a_1,a_2,\dots,a_{r-1}, a_{r} )
\]
where $a_r \geq 1$, and the $h$-vector is zero beyond degree $r$. Then we have (Proposition \ref{min hvtr})

\begin{itemize}
    \item[(a)] $a_1=3$ and $a_2\geq 4$;
    \item[(b)] the following possibilities hold:

\begin{itemize}
\item[\em (i)] $r=2$, $5 \leq a_2 \leq 6$;   

\item[\em (ii)] $r = 3$, $a_2 = 4$, $1 \leq a_3 \leq 4$; 

\item[\em (iii)] $r=3$, $a_2 = 5$, $1 \leq a_3 \leq 7$;

\item[\em (iv)] $r=3$, $a_2 = 6$, $1 \leq a_3 \leq 10$;

\item[\em (v)] $r \geq 4$, $a_k \geq 4$ for $2 \leq k \leq r-2$, and $a_{r-1} \geq 3$.
    
\end{itemize}
\end{itemize} 

\noindent {\bf We stress the fact that in all cases, $a_i \geq 4$ for $2 \leq i \leq r-2$.}

\item We also show (Lemma \ref{elim (3,2)}) that the $h$-vector $(1,3,4,4, \dots, 4, 3,2)$ cannot occur.

\item Now, to get a good bound for $|C_1 \cap C_2|$ using item \ref{item2} above, we need to see which $h$-vector $\underline{h}$ gives the biggest value of $g(\underline{h})$ among the $h$-vectors satisfying the constraints above. We do this using Corollary \ref{hvtr}, after converting the $h$-vectors to Hilbert functions.

The conclusion (see Remark \ref{a,b}) is that to achieve the maximum $g(\underline{h})$ (given $d$) under our restrictions, the $h$-vector $\underline{h}$ has the form 
\[
(1,3,\underbrace{4,\dots,4}_m,a,b) 
\]
(where we abuse notation and allow $b=0$). Then depending on congruence of $d\pmod{4}$, the biggest theoretical genus for a curve not lying on a cubic surface will come from an $h$-vector ending 
\[
(4,3,1),\; (4,1),\; (4,2) \mbox{ or } (4,3).
\]
We carry out the computation of $g(\underline{h})$ for such $\underline{h}$ in Corollary \ref{cor: genus bound explicit}.
This specific $h$-vector is used to define the {\it genus bound} $B_g(d_1,d_2)$.

\item After these calculations on the $h$-vector, we also need a result allowing us to conclude that $C_1 \cup C_2$ itself lies on an irreducible complete intersection surface in $\mathbb P^4$, by ``lifting" our conclusions about $\Gamma$. Using a theorem of Huneke and Ulrich, together with the LGP property, this is handled in Corollary \ref{lift ci22}.

\item Again invoking item \ref{item2} above, one of the big surprises in this paper is that an example exists of two smooth rational curves of odd degrees on a smooth complete intersection quartic surface whose union is ACM and which meet in $B(d_1,d_2)$ points. This is shown in Proposition \ref{lem: curves on del Pezzo quartic} and motivates the special exception in Conjecture \ref{conjecture b}.

\item In some situations (mostly when $d_1$ and $d_2$ are not far apart) we have $B_g(d_1,d_2) \leq B(d_1,d_2)$, proving most of Conjecture \ref{conjecture b} (see Theorem \ref{B minus Bg}). In other situations, $B_g(d_1,d_2)$ is  bigger. In Section \ref{sec: ACM curves} and Section \ref{sec: beyond Bg} we illustrate other methods that give additional partial results.

\end{enumerate}

\section{Proof of Theorem A, curves on a surface of degree 3 in $\mathbb P^4$}\label{sec: curves on a cubic}

In this section  we derive the bound for the number of intersection points of two reduced, irreducible, nondegenerate curves $C_1,C_2$ of degrees $d_1$ and $d_2$ under the assumption that they  lie on the same nondegenerate irreducible surface of degree 3 in $\PP^4$, and describe the curves that achieve the bound.

Note that a nondegenerate irreducible surface $S$ of degree 3 in $\mathbb P^4$ has minimal degree. 
By the classification of varieties of minimal degree (cf. \cite[Theorem 19.9]{HarrisBook}), we know that either $S$ is smooth, in which case it is isomorphic to the blow up of
$\PP^2$ at a single point $P$, embedded by the linear system $|2h-e|$ of conics through $P$; or $S$ is a cubic cone.
We begin with the case where $S$ is smooth (see Remark \ref{sing cubic} for the case where $S$ is a cubic cone). 

\begin{theorem} \label{bd on cubic}

Let $C_1, C_2$ be distinct reduced, irreducible 
curves of degrees $d_1, d_2 \geq 2$
lying on a smooth nondegenerate surface $S$ of degree 3 in $\mathbb P^4$, viewed as the blow up of $\PP^2$ at a point $P \in \mathbb P^2$.
Let $h$ be the divisor class coming from a line in $\mathbb P^2$ and let $e$ be the exceptional divisor on $S$.

Then $C_1$ and $C_2$ meet in at most $B(d_1,d_2)$ points, moreover the bound is sharp and it is achieved by taking the curves as follows:

\begin{itemize}

\item If $d_1\le d_2$ are both even, or if
$d_1$ is even and $d_2$ is odd,  
take $C_1$ and $C_2$ to be curves whose classes are $(d_1/2)h$ and $(d_2-1)h - (d_2-2)e$.

\item  If both $d_1$ and $d_2$ are odd, take curves $C_1$ and $C_2$ whose classes are
$((d_1+1)/2)h - e$ and $(d_2-1)h - (d_2-2)e$.
\end{itemize}
\end{theorem}
\begin{proof}
Let $\alpha h-\beta e$ be the class of a reduced, irreducible curve $C$ on $S$ with degree at least 2.
Thus $(\alpha h-\beta e)\cdot(2h-e)\geq 2$, or equivalently $2\alpha-\beta\geq 2$. Moreover, since $h$ and $h-e$ are nef on $S$,
we have 
$$(\alpha h-\beta e)\cdot h\geq 0\;\; (\text{so }\alpha\geq 0) \text{ and } 
(\alpha h-\beta e)\cdot (h-e)\geq 0\;\;(\text{so } \alpha\geq \beta).$$
In addition, since $C$ is reduced and irreducible, either $\alpha h-\beta e=e$ or $(\alpha h-\beta e)\cdot e\geq 0$ (so $\beta\geq 0$).
But $\alpha h-\beta e=e$ is ruled out since $e$ has degree 1.
In addition, $(\alpha h-\beta e)\cdot (h-e)> 0$ (hence $\alpha >\beta$), since otherwise $\alpha h-\beta e=m(h-e)$, and $|m(h-e)|$ is composed with a pencil,
so $C$ being reduced and irreducible means $m=1$, but $h-e$ has degree 1, so is ruled out.
Thus we have $\alpha > \beta\geq 0$.

Therefore we can write the class of $C$ as $ah+b(h-e)$ for $a>0$ and $b\geq0$
(where $\alpha=a+b$ and $\beta=b$).
We  note that $|ah+b(h-e)|$ for any such class contains a smooth, irreducible curve.
(Certainly $|ah|$ does since the members of this linear system are the same as plane curves of degree $a>0$.
And when $b>0$, $ah+b(h-e)$ is very ample since $2h-e=h+(h-e)$ is, hence the
general member of $|ah+b(h-e)|$ is smooth and irreducible by Bertini's theorem (see \cite[Theorem 8.18]{hartshorne}).

Let $a_ih+b_i(h-e)$ be the class of $C_i$ for $i=1,2$, with $a_i>0$ and $b_i\geq0$.
If $b_1=b_2=0$, then a general member of $|a_1h|$ meets a general member of $|a_2h|$ transversely
(because here we have a general plane curve of degree $a_1>0$ meeting a general plane curve of degree $a_2>0$).
And if $a_1, a_2>0$ and either $b_1>0$ or $b_2>0$ (say $b_1>0$), then again a general member of $|a_1h+b_1(h-e)|$ meets a
general member of $|a_2h+b_2(h-e)|$ transversely,
since a general member of $a_2h+b_2(h-e)$ is smooth and $a_1h+b_1(h-e)$ is very ample.
Thus the number of points of intersection of a general member $C_1$ of $|a_1h+b_2(h-e)|$ with
a general member $C_2$ of $|a_2h+b_2(h-e)|$ is exactly 
\begin{align*}
C_1\cdot C_2 &=(a_1h+b_1(h-e))\cdot(a_2h+b_2(h-e)) \\
&=
a_1a_2+a_1b_2+a_2b_1=a_1a_2+a_1(d_2-2a_2)+a_2(d_1-2a_1)\\
&=-3a_1a_2+a_1d_2+a_2d_1.
\end{align*}
Let $F(a_1,a_2)=-3a_1a_2+a_1d_2+a_2d_1$. Given $d_1$ and $d_2$, we want to find the maximum value of $F(a_1,a_2)$,
subject to $1\leq a_1\leq \lfloor\frac{d_1}{2}\rfloor$ and $1\leq a_2\leq \lfloor\frac{d_2}{2}\rfloor$
(since the smallest allowed value of $a_1$ is 1, and since $2a_1+b_1=d_1$, the largest allowed value of $a_1$ occurs when $b_1=0$ or 1, and is
thus $a_1= \lfloor\frac{d_1}{2}\rfloor$, and likewise for $a_2$).

Note for fixed $a_1$, that $F(a_1,a_2)=-3a_1a_2+a_1d_2+a_2d_1$ is linear in $a_2$, and likewise, for
fixed $a_2$, $F(a_1,a_2)=-3a_1a_2+a_1d_2+a_2d_1$ is linear in $a_1$. This means for a fixed $a_1$ that the maximum
of $F(a_1,a_2)$ occurs at an endpoint of $1\leq a_2\leq \lfloor\frac{d_2}{2}\rfloor$
(and likewise for a fixed $a_2$ the maximum
of $F(a_1,a_2)$ occurs at an endpoint of $1\leq a_1\leq \lfloor\frac{d_11}{2}\rfloor$).
Thus the maximum of $F(a_1,a_2)$
on the region $\{(a_1,a_2) : 1\leq a_1\leq \lfloor\frac{d_1}{2}\rfloor, 1\leq a_2\leq \lfloor\frac{d_2}{2}\rfloor\}$
occurs on the boundary, and on the boundary $F(a_1,a_2)$ achieves its maximum at one of the four corners,
$(1,1)$, $(1,\lfloor\frac{d_2}{2}\rfloor)$, $(\lfloor\frac{d_1}{2}\rfloor,1)$ or $(\lfloor\frac{d_1}{2}\rfloor,\lfloor\frac{d_2}{2}\rfloor)$.
So, for the various cases, we just need to check $F(a_1,a_2)$ at these four points.

When $d_1=2$ (or $d_2=2$), then the maximum occurs at $(a_1,a_2)=(1,1)$ and is $d_2-1$ when $d_1=2$ (and $d_1-1$ when $d_2=2$),
so the classes of $C_1$ and $C_2$ are $h$ and $h+(d_2-2)(h-e)$, resp., all as claimed
(and $h+(d_1-2)(h-e)$ and $h$ when $d_2=2$).
So hereafter we may assume $d_1,d_2>2$.

If both $d_1$ and $d_2$ are even with $d_1\leq d_2$, then the maximum is
$d_1(d_2-1)/2$ and occurs at $(a_1,a_2)=(d_1/2,1)$, so the curves are $C_1=(d_1/2)h$ and $C_2=h+(d_2-2)(h-e)$,
all as claimed.

Now assume $d_1$ is even and $d_2$ is odd. Then the maximum is again
$d_1(d_2-1)/2$ and occurs at $(a_1,a_2)=(d_1/2,1)$, so the curves are $C_1=(d_1/2)h$ and $C_2=h+(d_2-2)(h-e)$,
all as claimed.

Finally, assume $d_1$ and $d_2$ are both odd. Then the maximum is
$1+(d_1-1)(d_2-1)/2$ and occurs at both $(a_1,a_2)=((d_1-1)/2,1)$ and
$(a_1,a_2)=(1,(d_2-1)/2)$, so the curves $C_1$ and $C_2$ can be taken to be either
$((d_1-1)/2)h+(h-e)$ and $h+(d_2-2)(h-e)$, or $h+(d_1-2)(h-e)$ and $((d_2-1)/2)h+(h-e)$,
all as claimed.
\end{proof}

\begin{remark} \label{union not ACM}
    Question \ref{Q6} of the introduction asked whether two curves that meet in the maximum number of intersection points must have a union that is ACM. Theorem \ref{bd on cubic} gives a surprising negative answer. For instance, take $d_1=6, \ d_2=8$. Then $B(d_1,d_2) = 21$ and is achieved by curves whose classes are $3h$ and $7h-6e$. But one can check that the union, $10h-6e$, is not ACM. This can be checked by a computer algebra program or else one can observe that the hyperplane section of the cubic surface $S$ has divisor class $H = 2h-e$ so $10h-6e = 4H + (2h-2e)$. This is evenly linked to the divisor $(2h-2e)$ (\cite[Corollary 5.14]{KMMNP}). Since the latter is a pair of skew lines, and the ACM property is preserved under linkage, a curve in the divisor class $10h-6e$ is not ACM.
\end{remark}

Here we include two results to elucidate nondegeneracy of curves on a smooth cubic surface $S\subset\PP^4$,
where $S$ is as described in Theorem \ref{bd on cubic}. Thus $S$ is the blow up of $\PP^2$ at a point $P$
and its divisor class group has basis $h$, $e$, as above.

\begin{lemma} \label{degen red irred  C on cubic}
Let $C$ be a reduced irreducible curve on a smooth cubic surface $S\subset \PP^4$. Let $d$ be the degree of $C$.
Then $C$ is nondegenerate if and only if $d\geq4$.
\end{lemma}

\begin{proof}
Assume $C$ is nondegenerate. Since for a nondegenerate variety of degree $\delta$ and codimension $r$ in $\PP^n$ satisfies
$\delta\geq r+1$, applying this to $C$ we have $d\geq 4$.

Conversely, suppose $C$ is not nondegenerate. Thus there is a hyperplane $H$ such that $C\subset H$.
But as in Theorem \ref{bd on cubic}, the class of $H\cap S$ is $2h-e=h+(h-e)$. This has degree 3,
and since $C$ is a component of a hyperplane section, we get $\deg C\leq 3$.
\end{proof}

\begin{lemma} \label{degen red C on cubic}
Let $C$ be a reduced curve on a smooth cubic surface $S\subset \PP^4$. Let $d$ be the degree of $C$.
Then $C$ is nondegenerate if and only if either $d\geq4$, or $d=3$ with $[C]=h+e$ or with $[C]=3(h-e)$.
\end{lemma}

\begin{proof}
If $d\geq4$, then $C$ cannot be contained in a hyperplane, since a hyperplane section of $S$ has degree 3.
If $[C]=h+e$ or $[C]=3(h-e)$, then $d=3$. Suppose $[C]=h+e$. If $C$ were contained in a hyperplane, then
(since $\deg C=3$ and a hyperplane section of $S$ has degree 3) $C$ would have to be the hyperplane section;
i.e., $h+e=2h-e$, which is not the case. Thus $C$ is nondegenerate. Likewise, if $[C]=3(h-e)$, then
$C$ is nondegenerate.

For the converse, we prove the contrapositive of the converse.
I.e., we show that a curve $C$ of degree 3 or less whose class is not
$[C]=h+e$ or $[C]=3(h-e)$ is degenerate.
To do this, we first enumerate all reduced irreducible curves $A$ of degree 3 or less on $S$.
Let $ah-be$ be the class of $A$.
Since $h$ and $h-e$ are nef, $ah-be$ must satisfy $a\geq0$ and $a\geq b$.
Let $E$ be the curve whose class is $e$. Since $(ah-be)\cdot e=b$,
either $A=E$ or $b\geq 0$. Thus either $A=E$ or $a\geq b\geq 0$.
But if $a=b$, then the class of $A$ is $a(h-e)$, so $A$ is reduced and irreducible in this case
exactly when $a=1$. So the class of $A$ is either $e$, $h-e$ or $ah-be$ with $a>b\geq0$,
so $ah-be=\alpha h+b(h-e)$ where $a=\alpha+b$.

So the classes of reduced irreducible curves $A$ of degree 1 are $e$ and $h-e$.
The only class of a reduced irreducible curve $A$ of degree 2 is $h$.
The only class of a reduced irreducible curve $A$ of degree 3 is $2h-e$.
By Lemma \ref{degen red irred  C on cubic}, $A$ in each of these cases is degenerate.

So now we must consider reduced, reducible curves of degree at most 3.
A curve of degree 1 is reduced and irreducible so we need only consider
reduced but not irreducible curves of degree 2 or 3.

A reduced, reducible curve $C$ of degree 2 has two components, either one of which has class
$e$ and the other $h-e$, or both have class $h-e$.
Say its components are $h-e$ and $e$.
Any two such components span only a plane (being two intersecting lines), and hence $C$ in this case is degenerate.
If both are $h-e$, then $C+E$ has class $2h-e$ and is thus a hyperplane section and so is degenerate.

A reduced curve of degree 3 has either 2 irreducible components whose classes are thus $h$ and $e$, or $h$ and $h-e$,
or it has three irreducible components whose classes are thus $e$, $h-e$ and $h-e$ or all three have class $h-e$.

We saw above for $h+e$ and $3(h-e)$ that $C$ is nondegenerate. If it is $h$ and $h-e$,
then the class of $C$ is $2h-e$ so $C$ is a hyperplane section, so $C$ is degenerate.
If it is $e$, $h-e$ and $h-e$,
then the class of $C$ again is $2h-e$ so $C$ is a hyperplane section and $C$ is degenerate.
\end{proof}

\begin{remark} \label{sing cubic}
Here we consider the case of curves on $S$ in case the surface $S$ is a nondegenerate cubic cone.
Let the vertex of the cone be the point $P$. If we blow up $P$ we get a smooth ruled surface $S'$.
Let $e$ be the class of the exceptional curve $E$ for $P$ and let $f$ be the class of the ruling, so
$e$ and $f$ generate the divisor class group of $S'$. A hyperplane section of $S$ away from $P$
corresponds on $S'$ to a class of the form $ae+tf$ with $3=(ae+tf)^2=a^2e^2+2at$, $a=f(ae+tf)\geq0$ and $ae^2+t=e(ae+tf)=0$.
Thus $t=-ae^2$ so $3=a^2e^2-2a^2e^2=-a^2e^2$, hence $a=1$, $e^2=-3$ and $t=3$.

For a reduced curve $C_1 \subset S$ not through $P$, we can identify $C_1$ with its proper transform on $S'$,
so on $S'$ we have $[C_1]\cdot f\geq0$ (since $f$ is nef) and $C_1\cdot E=0$, hence $[C_1]=b_1(e+3f)$.
Let $C_2$ be a reduced curve on $S$ with no component 
in common with $C_1$, and let $D_2$ be its proper transform on $S'$ (so $D_2$ does not have $E$ as a component).
Then $|C_1\cap C_2|\leq C_1\cdot D_2$.
But $[D_2]=b_2e+c_2f$ with $D_2\cdot E\geq 0$ and $[D_2]\cdot f\geq0$ (so $c_2\geq 3b_2\geq0$), thus   
$d_1=\deg(C_1)=3b_1$ and $d_2=\deg(C_2)=D_2\cdot(e+3f)=(b_2e+c_2f)\cdot(e+3f)=c_2$.
We now have $C_1\cdot D_2=b_1c_2=d_1d_2/3$.
Since a general hyperplane section of $C_2$ will be transverse, we can replace $C_1$ by a curve
linearly equivalent to $C_1$ such that $|C_1\cap C_2| = C_1\cdot D_2=d_1d_2/3$.

Now assume $C_1$ and $C_2$ are reduced curves on $S$ through $P$ with no components 
in common, and let $D_1$ and $D_2$ be their proper transforms on $S'$ (so neither $D_1$ nor $D_2$ has $E$ as a component).
Then $|C_1\cap C_2|\leq D_1\cdot D_2 +1$, where the plus 1 comes from having $P$ in common.
(For example, if $C_1$ and $C_2$ are lines through $P$, then $|C_1\cap C_2| = D_1\cdot D_2 +1=f^2+1=1$.)

Here we have $[D_1]=b_1e+c_1f$ and $[D_2]=b_2e+c_2f$, with $d_1=c_1> 3b_1\geq0$ and $d_2=c_2> 3b_2\geq0$,
where we have $c_1>3b_1$ since $D_1\cdot E>0$ due to $P$ being on $C_1$ (and likewise 
$c_2>3b_2$). So we have  $D_1\cdot D_2=-3b_1b_2+c_1b_2+c_2b_1=-3b_1b_2+d_1b_2+d_2b_1$.

If $b_1=0$ (the case that $b_2=0$ is symmetric), then $D_1\cdot D_2=d_1b_2 < d_1c_2/3=d_1d_2/3$.
Thus we have $|C_1\cap C_2| \leq d_1b_2+1< \frac{d_1d_2}{3}+1$. 
If $b_2=0$, then $|C_1\cap C_2|  = d_1b_2+1=1$.
If $b_2>0$, then 
$D_2$ is very ample \cite[Corollary 2.18]{hartshorne}, so we may replace $D_2$ by a linearly equivalent
curve such that $D_1\cap D_2$ is transverse, so that replacing $C_2$ to match the new $D_2$ we have 
$|C_1\cap C_2| = d_1b_2+1< \frac{d_1d_2}{3}+1$.

Finally, assume $b_1>0$ and $b_2>0$. Thus $D_1$ and $D_2$ are very ample.
Then we have $|C_1\cap C_2| \leq -3b_1b_2+d_1b_2+d_2b_1+1$, with equality if we replace
$D_1$ and $D_2$ (and $C_1$ ad $C_2$ to match) so that $D_1$ and $D_2$ meet transversely.
Here we have $1\leq b_1<d_1/3$ and $1\leq b_2<d_2/3$.
Thus $-3b_1b_2+d_1b_2+d_2b_1=(d_1-3b_1)b_2+d_2b_1<(d_1-3b_1)d_2/3+d_2b_1=d_1d_2/3$.

In summary, 
if $C_1$ and $C_2$ are reduced curves on $S$ with no common components and at least one not containing $P$,
then $|C_1\cap C_2| \leq \frac{d_1d_2}{3}$ and this bound is sharp for each $d_1$ and $d_2$ (note 
that if $P\not\in C_1$, then $d_1$ must be a multiple of 3, and likewise for $C_2$ and $d_2$).
If $C_1$ and $C_2$ are reduced curves on $S$, both containing $P$ but with no common components,
then $|C_1\cap C_2| < \frac{d_1d_2}{3}+1$.

To get more precise bounds in the case that both curves $C_1$ and $C_2$ contain $P$,
recall that $[D_1]=b_1e+c_1f$ and $[D_2]=b_2e+c_2f$ where $d_1=c_1>3b_1\geq0$ and $d_2=c_2>3b_2\geq0$
with $d_1$ being $\deg(C_1)$ and $d_2$ being $\deg(C_2)$.
Define $i>0$ and $j>0$ by $d_1=c_1=3b_1+i$ and $d_2=c_2=3b_2+j$.
Then $D_1\cdot D_2=-3b_1b_2+c_1b_2+c_2b_1=-3(d_1-i)(d_2-j)/9+d_1(d_2-j)/3+d_2(d_1-i)/3=(d_1d_2-ij)/3$.
Thus $|C_1\cap C_2|\leq (d_1d_2-ij)/3 +1$. 
Clearly $D_1\cdot D_2$ increases as $ij$ decreases. For fixed $d_1$ and $d_2$,
there is a unique choice of $b_1$ and $b_2$ so that $1\leq i,j\leq 3$.
Any other choice of $b_1$ and $b_2$ giving the same values for $d_1$ and $d_2$
would force $i$ or $j$ to be larger. Assume $b_1$ and $b_2$ have been chosen so
that $1\leq i,j\leq 3$ and hence $ij$ is as small as possible.
If $b_1=b_2=0$, then $d_1d_2-ij=0$ so $|C_1\cap C_2|=1$.
If $b_1>0$ (or $b_2>0$, resp.), then $D_1$ (or $D_2$, resp.) is ample so meets $C_2$ (or $C_1$, resp.) 
transversely, so $|C_1\cap C_2| = (d_1d_2-ij)/3 + 1$.
Note that $(d_1d_2-ij)/3 + 1\leq d_1d_2/3$ unless $(i,j)\in\{(1,1), (1,2), (2,1)\}$.
In these cases we have, resp., $\frac{d_1d_2-ij}{3} + 1=\frac{d_1d_2+2}{3}, \frac{d_1d_2+1}{3}, \frac{d_1d_2+1}{3}$.
\end{remark}

\begin{remark}
Our main focus for the rest of the paper is to show to what extent the bound from Theorem~\ref{bd on cubic} holds for intersections of curves not necessarily on a cubic surface. Our motivation comes from the known result for curves in $\PP^3$ \cite{diaz}, \cite{Giuffrida86}, where a bound is given for the number of intersection points of two irreducible curves on a quadric surface, and it is further shown that if the curves do not lie on a common quadric surface then the bound is never achieved. Theorem \ref{bd on cubic} gives our analog for the former result.

The rest of the paper will be geared toward a better understanding of what happens off the cubic surface, and to proving special cases of the conjecture.
\end{remark}

\subsection{Smooth, irreducible curves not in a cubic surface with high intersection numbers}
\label{sec: quartic}
Motivated by the constructions on the cubic surface in the previous section, we now construct certain pairs of curves on the del Pezzo surface of degree 4 with high intersection numbers. This gives examples of curves of degree $d_1, d_2$ not contained on a cubic surface meeting in $B(d_1,d_2)$ points.  These curves will later serve as key examples in Theorems~\ref{t: bd in P4 same deg} and~\ref{case d neq d'}. We note that these curves are not only irreducible but in fact smooth.

Let $f\colon S\to\PP^2$ be the blow up of $\PP^2$ at five general points $P_1,\ldots,P_5$ with exceptional divisors $e_1,\ldots,e_5$ and let $h=f^*\mathcal{O}_{\PP^2}(1)$.
The linear system $L=3h-e_1-\cdots-e_5$ embeds $S$ as a quartic surface into $\PP^4$.

\begin{proposition}\label{lem: curves on del Pezzo quartic}
For odd integers $d_1=2k+1$ and $d_2=2\ell+1$ there exist smooth curves $C_1, C_2$ on the del Pezzo quartic surface in $\PP^4$ such that   $$\deg(C_1)=d_1, \;\deg(C_2)=d_2\;\;
\mbox{ and }\;\; |C_1\cap C_2|=\frac{1}{2}(d_1-1)(d_2-1)+1=B(d_1,d_2).$$ 
\end{proposition}
\begin{proof}
We consider the following two linear systems on $S$:
$$L_1=(2k+1)h-e_1-ke_2-ke_3-ke_4-(k+1)e_5\; \text{ and }
L_2=(\ell+1)h-\ell e_1-e_2-e_3.$$

Since a singularity of multiplicity $r$ lowers the geometric genus provided by the genus-degree formula by at least $r(r-1)/2$ (see \cite[page 54]{Semple_Roth_1985}, all singularities of curves in $\PP^2$ corresponding to members of $L_1$ and $L_2$ must be ordinary and the curves on $S$ are rational. Furthermore we have
$$L.L_1=2k+1 \mbox{ and }L.L_2=2\ell+1,$$
so the linear system $L_1$ viewed on the image of $S$ in $\PP^4$ consist of curves of degree $d_1=2k+1$, while $L_2$ contains curves of degree $d_2=2\ell+1$. We observe also right away that
$$L_1.L_2=2k\ell+1  = \frac{1}{2}(d_1-1)(d_2-1) + 1.$$
We want to ensure that some members of the two linear systems do intersect in that number of distinct points. We first show that this can be achieved by some reducible curves in $L_1$ and $L_2$. 

Indeed, let $D_1\subset\PP^2$ be the union of $k$ general (hence irreducible) conics in the pencil of conics through $P_2,\ldots,P_5$ and the line through $P_1$ and $P_5$. Then $F_1$, the proper transform of $D_1$, is a member of $L_1$.

Similarly, let $D_2\subset \PP^2$ be the union of $\ell$ general (hence omitting the points $P_2$ and $P_3$) lines through $P_1$ and the line through $P_2$ and $P_3$. Then $F_2$, the proper transform of $D_2$ is a member of $L_2$.

Away from the points $P_1,\ldots,P_5$, the plane curves $D_1$ and $D_2$ intersect in $2k\ell$ points -- coming from the $\ell$ lines through $P_1$ and the $k$ conics through $P_2,\ldots,P_5$ -- along with one additional intersection point of the lines $P_1P_5$ and $P_2P_3$. All these intersection points persist on the surface $S$ and in its embedding into $\PP^4$.

We claim now that there are \emph{irreducible} members $C_1$ of $L_1$ and $C_2$ of $L_2$ which maintain the same number of intersection points. 

By Bertini's theorem \cite[Theorem 5.3]{Kleiman98} the general element of a linear system
can be reducible if and only if
\begin{enumerate}
    \item there is a fixed component, or 
    \item the linear system is composed with a pencil.
\end{enumerate}
To illustrate the second possibility, consider the linear system
of plane cubics with one fixed triple point.
All the elements split in the union of 3 lines, even if there are
no fixed components. The problem is that this linear system is composed
with the pencil of lines through a fixed point, and every element
is the union of three lines of the pencil.

To exclude that a linear system is composed with a pencil it is sufficient
to prove that the subsystem of elements through a (fixed) general point
has no fixed components.

We apply now this line of thought for the linear system $L_1$.
Let $D_1'$ be the union of the conic through all the points $P_1,\ldots,P_5$, $(k-1)$ general conics through $P_2,\ldots,P_5$ and a general line through $P_5$.

Observe that the construction of divisors like $D_1$ and $D'_1$ shows the existence of elements of $L_1$ passing through a general $P\in S$ and sharing no components. Indeed, in the case of $D_1$ we can take one of the general conics through $P_2,\ldots, P_5$ to pass through $P$. In case of $D_1'$ we can take the general line through $P_5$ as the line through $P$. These two curves have no common components. So a general element $\Gamma_1$ of $L_1$ is irreducible.

Turning to the linear system $L_2$, let $D_2'$ be the union of lines $P_1P_2$, $P_1P_3$, $(\ell-2)$ general lines through $P_1$ and another general line in $\PP^2$.

In the case of $D_2$ we can take one of the general lines through $P_1$ to pass through $P$ and in case of $D_2'$ it could be the general line through $P$. Hence, as before, the general element $\Gamma_2$ of $L_2$ is irreducible.

Taking as $C_1$ and $C_2$ the proper transforms of the aforementioned irreducible plane curves $\Gamma_1, \Gamma_2$, we conclude the argument by observing that intersecting in $2k\ell+1$ distinct points is an open condition.

Moreover, $C_1$ and $C_2$ must, in fact, be smooth. This follows from the construction: their arithmetic genus is zero, and the blowup resolves all singularities of the initial plane curves.
\end{proof}

\section{Proof of Proposition C, one curve on a surface of degree $3$ in $\PP^4$}\label{sec: proof of proposition C}
First we will show that nondegenerate and irreducible curves in $\PP^4$ of degree $\leq 5$ must be contained in a cubic surface. Any irreducible curve of degree $\leq 3$ is degenerate and the claim is clear for irreducible curves of degree $4$, so we need to take care of irreducible curves of degree $5$. While this fact is probably well known to specialists we record it here as we are unaware of a reference.

\begin{lemma} \label{quintic}
A nondegenerate, irreducible curve $C \subset \PP^4$ of degree $5$ lies in a cubic surface.
\end{lemma}

\begin{proof}
First assume that the quintic curve $C$ is rational (arithmetic genus 0). 
Then $C$ is a projection of an embedding of $\PP^1$ with the full linear system of degree 5, hence it is a projection of a rational normal curve $C'$ in $\PP^5$.  
Note that the 4-secant variety of $C'$ covers $\PP^5$, since curves are never defective (this is a general classical consequence of the Terracini's Lemma, see, e.g., \cite{CCi}). Hence a general point of $\PP^5$ lies in some $\PP^3$ which is 4-secant to $C'$.
When we project $C'$ from a point $P$, 4-secant $\PP^3$'s to $C'$ passing from $P$
map to planes in $\PP^4$ which meet $C$ in four points. So, 4-secant planes $\pi$
to $C$ in $\PP^4$ exist. When the projection from $C'$ to $C$ is not general, the existence of $4$-secant planes  to $C$ follows by semicontinuity.

Thus, fix a plane $\pi$ which meets $C$ in degree at least $4$, and fix two 
general points $Q,Q'$ of $\pi$. By Riemann-Roch, the space of  quadric hypersurfaces containing $C$
has dimension at least $4$. Thus there exists a pencil of quadrics that contain $C$ and $Q,Q'$. 
These quadrics must contain $\pi$, because they contain 6 points of $\pi$, two of which are general.
Since two general quadrics of the pencil meet in a quartic surface $S$ which
contains both $\pi$ and $C$, then $C$ lies in the residual $S'$ of $\pi$ in the complete
intersection $S$. The surface $S'$ has degree $3$.

If the quintic curve $C$ is elliptic (arithmetic genus 1), then it is contained in a $5$-dimensional family of
quadric threefolds, by Riemann-Roch.  Thus, for any $3$-secant plane $\pi$ to $C$, and for a general choice of three points $Q,Q',Q''\in\pi$, there exists a pencil of quadrics containing $C$ and $Q,Q',Q''$. All the
quadrics of the pencil must contain $\pi$. We conclude as above.
\end{proof}

We conclude this section by considering the situation when one of the curves in $\PP^4$, say $C_1$ is contained in a cubic surface and the other curve, $C_2$ isn't. By the way of preparation we have the following statement.

\begin{lemma} \label{C int S}
    Let $C \subset \PP^4$ be a reduced, irreducible, nondegenerate curve of degree $d$ and let $S$ be an irreducible cubic surface not containing $C$. Then $C$ and $S$ meet in at most $2d-1$ points.  If $C$ is not rational then $C$ and $S$ meet in at most $2d-2$ points. For any odd $d$, a rational curve of degree $d$ and a cubic surface $S$  exist with $|C \cap S| = 2d-1$.
\end{lemma}

\begin{proof}
We first claim that $|S \cap C| \leq  2d$. Since $S$ is cut out by three independent quadrics, a general quadric $F$ containing $S$ does not contain $C$, so it meets $C$ in at most $2d$ points. This proves the claim.

Suppose $| C \cap S| = 2d$ and let $Z$ be this set of $2d$ points. Let $P$ be a general point of $C$ and let $F$ be any element of $[I_S]_2$ (thought of as a hypersurface) vanishing at $P$. Since $C \cap F$ contains $Z$ as well as $P$, we get that $F$ must contain all of $C$. But there is a pencil of such $F$, so two general such $F$ give a complete intersection of degree 4, linking $S$ to a plane and containing $C$. Since $C$ is nondegenerate, $C$ must lie on $S$, a contradiction. Thus $| C \cap S | \leq 2d-1$.

Suppose that $|C \cap S| = 2d-1$ and let $Z$ be this set of $2d-1$ points. The general element of $[I_S]_2$ does not contain $C$ since $C$ is not contained in $S$. We again consider a general point $P \in C$ and the pencil of elements of $[I_S]_2$ containing $P$. As before, using liaison, if every element of this pencil contains $C$, we get that $C$ must lie on $S$, a contradiction. Thus the general element of this pencil cuts out the union  of $Z$ and $P$. But $P$ was a general point. Thus the linear system of quadrics containing $S$, restricted to $C$, has $Z$ as a base locus and the residual is a $g_1^1$, i.e., $C$ is rational.

Finally, if $D$ is  any rational quintic curve, we know by Lemma \ref{quintic} that $D$ lies on a cubic surface $S$. Let $d$ be odd. Then $B(5,d) = 2d-1$. In Theorem \ref{case d neq d'} below we will give a construction  of such curves $C,D$ achieving $|C \cap D| = B(5,d)$ and not  lying on a common cubic surface. But then also $|C \cap S| = 2d-1$.
\end{proof}
We obtain immediately the following result.
\begin{corollary} \label{cor of C int S}
    Let $C_1, C_2$ be reduced, irreducible, nondegenerate curves in $\PP^4$. Assume that $C_1$ lies on a cubic surface $S$ and $C_2$ does not lie on $S$.Then $|C_1 \cap C_2 | \leq 2d_2-1$. If $C_2$ does not have arithmetic genus zero then $|C_1 \cap C_2| \leq 2d_2-2$.
\end{corollary}

We study now curves of low degree with respect to \Cref{conjecture b}.
\subsection{The case $d_1 = 4$}

For $d_1=4$ we have $B(4,d_2) = 2(d_2 -1) = 2d_2-2$.
If $C_2$ lies on the same surface of degree $3$ as $C_1$, then we know by Theorem \ref{bd on cubic} that this bound is sharp.

So let us assume that $C_2$ does not lie on the same cubic surface as $C_1$. Then \Cref{C int S} shows that $B(4,d_2)$ is an upper bound for $|C_1 \cap C_2|$ and the only question is its sharpness. 

Suppose first that we also have $d_2 = 4$, so $B(4,4) = 6$. To see that this bound is achieved, let $C_1$ be an irreducible rational normal curve and link it by a general complete intersection in $I_{C_1}$ of type $(2,2,2)$. The residual, $C_2$, is again a rational normal curve. One checks that $|C_1 \cap C_2| = 6$.  However, $C = C_1 \cup C_2$ does not lie on a cubic surface since such a surface is minimally generated by three quadrics, and $I_C$ is minimally generated by three quadrics that define a curve, not a surface. 

Now suppose $d_2 \geq 5$. The next Lemma shows that $B(4,d_2)$ is never achieved. Note that a nondegenerate, irreducible curve of degree $4$ in $\PP^4$ is a variety of minimal degree, hence the rational normal curve.

\begin{lemma} \label{rnc}
Let $C_1$ be a rational normal curve in $\PP^4$ and let $C_2$ be a reduced, irreducible, nondegenerate curve of degree $d_2 \geq 5$ not on a common cubic surface with $C_1$. Then 
$$|C_1 \cap C_2|  < B(4,d_2)=2d_2-2.$$
\end{lemma}

\begin{proof}
By Corollary \ref{cor of C int S}, we have $|C_1 \cap C_2| \leq B(4, d_2)$. To complete the argument, it remains to exclude the case of equality. Arguing by contradiction, we assume that 
$$|C_1 \cap C_2| = 2d_2-2.$$ 
We know that $\dim [I_{C_1}]_2 = 6$. Let $P_1, P_2, P_3$ be general points of $C_2$. By B\'ezout's theorem, any element of $[I_{C_1}]_2$ vanishing at $P_1, P_2, P_3$ also contains all of $C_2$ (because it contains the $2d_2-2$ intersection points with $C_1$ and $P_1,P2,P_3$). 

Thus $\dim [I_{C_1 \cup C_2}]_2 \geq 6-3 = 3$. There is no regular sequence of length 3 since $\deg (C_1 \cup C_2) \geq 9$. Hence the base locus is at least 2-dimensional, and since every element is irreducible it is not 3-dimensional. Since the union $C_1 \cup C_2$ does not lie on a surface of degree $\leq 3$, the union must lie on an irreducible complete intersection quartic surface. But now $\dim [I_{C_1 \cup C_2}]_2 \geq 3$ defines a complete intersection quartic surface. This is a contradiction.
\end{proof}

\subsection{The case $d_1 = 5$}

Let $C_1 \subset \PP^4$ be a reduced, irreducible, nondegenerate curve of degree 5. By Lemma~\ref{quintic} $C_1$ lies on an irreducible cubic surface. 

We may assume that $d_2 \geq 5$. Then 
\[
B(5,d_2) = 
\left \{
\begin{array}{ll}
2d_2 & \hbox{if $d_2$ is even}; \\
2d_2-1 & \hbox{if $d_2$ is odd}.
\end{array}
\right.
\]
We are once again interested in the sharpness of this bound. The following lemma shows that the answer depends on the parity of $d_2$.
\begin{lemma} \label{bd for quintic}
Let $C_1, C_2$ be reduced, irreducible, nondegenerate curves in $\PP^4$, and assume $\deg C_1 = 5$ and $\deg C_2 = d_2 \geq 5$. Assume that $C_1$ and $C_2$ do not lie on a common cubic surface.  If $d_2$ is even then $|C_1 \cap C_2| < B(5,d_2) = 2d_2$. If $d_2$ is odd then $|C_1 \cap C_2| \leq B(5,d_2) = 2d_2-1$, and equality can occur only if $C_2$  has arithmetic genus 0.
\end{lemma}

\begin{proof}
    This follows immediately from Lemma \ref{C int S}.
\end{proof}

In the next Remark we show $B(5,5)$ can be reached even if both curves are not contained in the same cubic surface.
\begin{remark} \label{low deg th 5.9}
For $d_1 = d_2 = 5$ we have $B(5,5) = 9$.
The construction given in the  proof of Theorem \ref{t: bd in P4 same deg} gives a pair of quintic curves $C_1$ and $C_2$, both of arithmetic genus zero, meeting in 9 points. While each of them must be contained in some cubic surface by \Cref{quintic} however their union $C_1 \cup C_2$ does not lie on a cubic surface. Indeed, the $h$-vector of the hyperplane section of $C_1 \cup C_2$ is $(1,3,4,2)$ so even though $C_1 \cup C_2$ is ACM (because it achieves the maximum number $B_g(5,5)$ of intersection points, see \Cref{genus formula}), we get that $C_1 \cup C_2$ lies on at most a pencil of quadrics, and so it does not lie on a cubic surface.
\end{remark}

\section{The role of the genus and $h$-vectors}
\label{sec: genus bd}

\subsection{Genus bounds, first view}
In this paper, we adopt a classical approach to bounding $|C_1 \cap C_2|$ by analyzing the Hilbert function of a general hyperplane section of $C = C_1 \cup C_2$ and computing the arithmetic genus of $C$. This section outlines the general strategy. However, as always, the devil is in the details. In the remainder of the paper, we investigate specific constraints on the Hilbert function that yield cleaner genus formulas and sharper bounds on the number of intersection points in terms of the degrees of the intersecting curves.

A key component of our approach is the fundamental relationship between the number of intersection points $|C_1 \cap C_2|$ and the genera of $C_1$, $C_2$, and their union $C_1 \cup C_2$, captured in the following result.

\begin{proposition}[{\cite[Proposition 4]{MR1985}}] \label{rosa fact}
    Let $C = C_1 \cup C_2 \subset \PP^n$ be the union of two curves. Let $g, g_1, g_2$ be their respective arithmetic genera, and assume that the curves meet in $|C_1\cap C_2|$ points (counted without multiplicity). Then
 \[
    g \geq g_1 + g_2 + |C_1\cap C_2| -1,
    \]
    so 
    \[
    |C_1\cap C_2| \leq g - g_1 - g_2 + 1.
    \]
\end{proposition}

\begin{proof}
 In \cite[Proposition 4]{MR1985} the author works in $\PP^3$  and counts $|C_1\cap C_2|$ as the length of the scheme defined by the intersection rather than as simply a number of points. However, her proof works also in our setting.
\end{proof}

The connection between the genus and the Hilbert function (again very well known but reviewed here) is given as follows.

\begin{proposition}[cf. {\cite[Proposition 1.4.2]{MiglioreBook}}] \label{genus formula}
Let $C \subset \PP^n$ be an equidimensional curve of degree $d$ with Hilbert function $h_C(t)$. Let $\Gamma$ be a general hyperplane section of $C$, with Hilbert function $h_\Gamma(t)$. Let $\ell \gg 0$, so that $\Delta h_C(\ell) = d$. Then the arithmetic genus $g$ of $C$ satisfies
\[
g = \sum_{i=1}^\ell [d - h_\Gamma(i)] - k
\]
where $k$ is the vector space dimension of a certain submodule $K$ (identified in the proof) of the Hartshorne-Rao module $M(C)$ of $C$. Moreover,  $C$ is ACM if and only if $k =0$. 
\end{proposition}

\begin{proof}
Part of this result can also be found in \cite[Lemma 4.1]{hartshorne2015}. 
The Hartshorne-Rao module of $C$ is $M(C) = \bigoplus_{t \in \mathbb Z} H^1(\PP^n, \mathcal I_C(t) )$.
For a general linear form $L$  we have an induced multiplication $\times L\colon M(C) (-1) \rightarrow M(C)$. Let $K = \bigoplus_t K_t$ be the kernel of $\times L$.
For any integer $t$  we then have an induced homomorphism $\tau_t \colon [M(C)]_{t-1} \rightarrow [M(C)]_t$ from the degree $t-1$ component to the degree $t$ component. We set $k_t$ to be the dimension of the kernel $K_t$ of this homomorphism and we have $k = \sum_{t \geq 0} k_t$.

    More precisely, for any $t$ we have the exact sequence
    \[
    0 \rightarrow [I_C]_{t-1} \rightarrow [I_C]_t \stackrel{\sigma_t}{\longrightarrow} [I_{\Gamma|H}]_t \rightarrow H^1(\mathcal I_C(t-1)) \stackrel{\tau_t}{\longrightarrow} H^1(\mathcal I_C(t)) \rightarrow \cdots
    \]
    where $I_{\Gamma|H}$ is the homogeneous ideal of $\Gamma$ viewed as a subscheme of $H = \PP^{n-1}$. Note that if $C$ is not ACM then for at least one $t$, the restriction map $\sigma_t$ is not surjective.  Thus if $C$ is not ACM then $k_t > 0$ for at least one $t$ and we have
    \[
    \dim [I_C]_t - \dim [I_C]_{t-1} = \dim [I_{\Gamma|H}]_t - k_t \leq \dim [I_{\Gamma|H}]_t
    \]
    where the inequality is strict in at least one degree. Passing to the Hilbert functions of $C$ and of $\Gamma$, it follows that 
    \[
    \Delta h_C(t) = h_\Gamma (t) + k_t.
    \]
 But then Riemann-Roch (see \cite[page 29]{MiglioreBook}) gives for $\ell \gg 0$
    \[
    g = d \ell +1 - \sum_{i=0}^\ell \Delta h_C (i) = d \ell  - \sum_{i=1}^\ell h_\Gamma (i) - \sum_{i=0}^\ell k_i = \sum_{i=1}^\ell [d - h_\Gamma(i)]  -k
    \]
    as desired. 
    Finally, it is a standard fact (see for instance \cite{MiglioreBook}) that $C$ is ACM if and only if $M(C) =0$ if and only if $k=0$.
\end{proof}

\begin{example}
    If $C \subset \PP^4$ is ACM with $h$-vector $(1,3,4,4,2)$ then its general hyperplane section $\Gamma$ has Hilbert function $(1,4,8,12,14,14,\dots)$ so $d=14$ and the arithmetic genus is
    \[
    (14 - 4) + (14 - 8) + (14 - 12) = 18.
    \]
    If $C$ is not ACM but its general hyperplane section has the same $h$-vector then the genus of $C$ is strictly less than 18.
\end{example}

\begin{remark} \label{formal calc}
    Given an $h$-vector $\underline{h} = (1,h_1,h_2,\dots,h_s)$, Proposition \ref{genus formula} gives rise to a purely numerical formula that we can use without regard to any specific curve associated to the $h$-vector. First of all, we will ignore the invariant $k$ in Proposition \ref{genus formula}. Then we compute the ``integral"
\[
(1, 1+h_1, 1+h_1+h_2, \dots, 1+h_1+\cdots + h_s, \dots).
\]
If we formally call this sequence $h_\Gamma$ and set $d = 1+h_1+\cdots+h_s$, we formally obtain an invariant 
\[
g(\underline{h}) = \sum_{i=1}^\ell [d-h_\Gamma(i)].
\]
This formal computation will be used in the next section, where we will associate to a degree $d$ a specific ``extremal" $h$-vector, and apply this calculation to that $h$-vector and obtain an invariant $g(d)$ that will be important for our results.
\end{remark}

\begin{example}

One thing that we would like to stress is the difference between the role of the $h$-vector of $C$ and the $h$-vector of the general hyperplane section, $\Gamma$, of $C$ in the above discussion. Note that these $h$-vectors agree if and only if the curve is ACM.

A surprising fact is that examples exist of non-ACM curves in $\PP^3$ whose Hilbert functions agree with Hilbert functions of  ACM curves, so some confusion could arise. See \cite{MN11} for an extensive study of this behavior, but in fact the first example was due to Joe Harris much earlier (the paper \cite[Section 2]{MN11} included Harris's example with his permission). Harris produces two curves. The first, $C_1$, is a curve of degree 10 produced as the residual, in a complete intersection of two quartics, of a curve of type $(2,4)$ on a quadric surface. This curve is not ACM. The second, $C_2$, is defined by the ideal of maximal minors of a general $4 \times 5$ matrix of linear forms, which is ACM. (His construction was actually slightly more complicated, but this suffices for our purposes here.)

Since $C_1$ is in the liaison class of two skew lines, its Hartshorne-Rao module is 1-dimensional (as a vector space) so $k=1$ in Proposition \ref{genus formula}. The Hilbert functions of $C_1$ and $C_2$ are the same: $(1,4,10, 20, 30, \dots)$. Since they have the same Hilbert polynomial, $10t-10$, they both have arithmetic genus 11. 

In both cases the first difference of the Hilbert function is $(1,3,6,10, 10, 10, \dots)$.
However, one can check that the general hyperplane section $\Gamma_1$ of $C_1$ has Hilbert function $(1,3,6,9,10,10,\dots)$, while the general hyperplane section $\Gamma_2$ of $C_2$ has Hilbert function $(1,3,6,10,10,10,\dots)$. Then we confirm that $C_1$ has genus
\[
(10-3) + (10-6) + (10-9)    - 1 = 12-1 = 11
\]
and $C_2$ has genus
\[
(10-3) + (10-6)   = 11.
\]
If $C_1$ had been ACM with the given Hilbert function for the hyperplane section, its genus would have been 12 because the Hartshorne-Rao module would have been zero, so $k$ would also be zero in Proposition \ref{genus formula}.
\end{example}

\begin{example}
    Let $C_1,C_2$ be ACM curves in $\PP^4$ of degree 18 with $h$-vectors $(1,3,3,3,3,3,2)$ and $(1,3,4,4,4,2)$ respectively. Then our formulas give the genus of $C_1$ to be 40 and of $C_2$ to be 32. This reflects the general philosophy (not a hard and fast rule) that to maximize the genus, for fixed degree, we need to keep the $h$-vector as low as possible (within other constraints such as linear general position (LGP)). 
\end{example}

\begin{corollary} \label{hvtr}
    Let $C_1$ and $C_2$ be nondegenerate ACM curves in $\PP^n$ of the same degree and with arithmetic genera $g_1, g_2$. Let $\Gamma_1, \Gamma_2$ be the general hyperplane sections of $C_1$ and $C_2$ respectively. If $h_{\Gamma_1}(t) \leq h_{\Gamma_2}(t)$ for all $t$ then $g_1 \geq g_2$. If, in addition, $h_{\Gamma_1}(t) < h_{\Gamma_2}(t)$ for some $t$ then $g_1 > g_2$.
\end{corollary}

\begin{proof}
It is immediate from Proposition \ref{genus formula}.
\end{proof}

\begin{example}
Consider the following example when $d=16$:
\[
\begin{array}{lll}
\underline{h} = (1,3,5,4,3) & \Rightarrow & g(\underline{h}) = 22; \\
\underline{h} = (1,3,4,4,4) & \Rightarrow & g(\underline{h}) = 24; \\
\underline{h} = (1,3,4,4,3,1) & \Rightarrow & g(\underline{h}) = 25.
\end{array}
\]

\end{example}

Now we introduce the {\it genus bounds} that we will use.
Let $C_1$ and $C_2$ be reduced, irreducible, nondegenerate curves in $\PP^4$. Let $C = C_1 \cup C_2$. Note that $C$ is not necessarily ACM even if $C_1$ and $C_2$ are, and $C$ could be ACM even if $C_1$ and/or $C_2$ are not. Let $H$ be a general hyperplane and let $\Gamma = C \cap H$. Of course $\Gamma$ {\it is} necessarily ACM, so its $h$-vector is the Hilbert function of the artinian reduction of the coordinate ring of $\Gamma$.

\begin{remark} \label{genus bound}  Given $d_1$ and $d_2$ with $d = d_1 + d_2$, we note that if $g$ is the arithmetic genus of $C = C_1 \cup C_2$ then from Proposition \ref{genus formula} and Proposition \ref{rosa fact} (taking $g_1 = g_2 = 0$) we have
\[
|C_1 \cap C_2| \leq g +1 \leq \sum_{i=1}^\ell [d - h_\Gamma(i)] +1
\]
where the second inequality  is necessarily strict if $C$ is not ACM, and the first is necessarily strict if $C_1$ and $C_2$ do not both have arithmetic genus zero. This gives the largest theoretically possible genus for $C_1 \cup C_2$, and thus the largest theoretically possible number of intersection points. 
\end{remark}

\subsection{Basic results on $h$-vectors}
\label{sec: classical}

This section and the next we recall and develop some geometric tools that we will use in the rest of the paper.

\ Recall that the Uniform Position Lemma of Harris (\cite[page 197]{harris}) says the following. {\it Let $C$ be a reduced, irreducible, nondegenerate curve in $\PP^n$. Let $\Gamma$ be a general hyperplane section of $C$. If $\Gamma$ imposes $t$ conditions on the linear system $|\mathcal O_{\PP^{n-1}}(l)|$ then any $t$ points of $\Gamma$ do.} A set of points with this property is said to have the {\it Uniform Position Property (UPP)}.

Now we turn our attention to the general hyperplane section of the union of two irreducible curves in $\PP^4$. Although this set of points may not have the UPP, we will now show that they still are in linear general position (LGP). This fact will be important in the sequel.

\begin{lemma} \label{avoid double tangent}
Let $C_1, C_2$ be nondegenerate, reduced, irreducible curves in $\PP^4$. Then $C_1$ and $C_2$ have at most finitely many tangent lines in common. In particular, the tangent line to $C_1$ at a general point is not tangent to $C_2$.
\end{lemma}

\begin{proof}

Recall that if $C$ is an irreducible curve in $\PP^2$, and $C'$ is the Zariski open subset of smooth points of $C$ then we get a dual curve $D$ in $(\PP^2)^\vee$ that is not closed but is 1-dimensional, and the points of $D$ correspond to the tangent lines of $C'$.

 Let $D'$ be the smooth points of $D$. The dual of $D'$ is Zariski open in $C'$. That is, by taking the Zariski closure, {\it we recover $C$ from the dual of $C'$}. So a different curve from $C$ can't have the same dual curve.
 Therefore two irreducible plane curves have at most finitely many tangent lines in common.

  Now let $C_1$ and $C_2$ be nondegenerate, reduced, irreducible curves in $\PP^4$ and let $\lambda$ be a general line in $\PP^4$. Then the projection from $\lambda$, $\pi_\lambda$,  to a general plane $\PP^2$ gives two plane curves, say $Y_1 = \pi_\lambda (C_1)$ and $Y_2 = \pi_\lambda (C_2)$. The image of a tangent line to $C_1$ at a general point of $C_1$ maps under the projection to a tangent line of $Y_1$. Similarly for $C_2$ and $Y_2$.

This means that $C_1$ and $C_2$ have at most a finite number of common tangent lines.
\end{proof}

\begin{lemma} \label{lgp rmk}
    Let $C$ be the union  of two nondegenerate curves $C_1$ and $C_2$ in $\PP^4$, both  of which are reduced and irreducible, of degrees $d_1, d_2$ respectively. We do not assume that the curves have the same degree or that they are disjoint. Let $\Gamma = \Gamma_1 \cup \Gamma_2$, where $\Gamma_i = C_i \cap H$ and $H$ is a general hyperplane. 

    \begin{itemize}
    \item[\rm (i)] If $G_i$ is the monodromy group of $\Gamma_i$ as $H$ varies, $i = 1,2$, then the monodromy group of $\Gamma$ is $G_1 \times G_2$.  More precisely, $G_1 \times G_2 = S_{d_1} \times S_{d_2}$, the direct product of the symmetric groups. 
        \item[\rm (ii)]  The set $\Gamma$ is in LGP. 
        \item[\rm (iii)]  Suppose that $\Gamma$ lies on a unique curve $Y$ of a certain degree and suppose that $Y = Y_1 \cup Y_2$. Then both $Y_1$ and $Y_2$ are irreducible and (up to interchanging $C_1$ and $C_2$) $\Gamma_i \subset Y_i$ for $i = 1,2$.
        \end{itemize}
    \end{lemma}

\begin{proof}

Part (i)  is immediate from Exercise 11.4 of  \cite{EHbook}. Its proof is a small modification of the argument for the fact that the monodromy group  of the general hyperplane section of just one irreducible curve is the full symmetric group; see   \cite[Lemma, page 698]{HarrisDuke}, \cite[Uniform Position Lemma, page 197]{harris}, \cite[Remark 1.7]{Rathmann},  \cite[pages 111--113]{ACGH} and especially \cite[Chapter 11]{EHbook}. The key point for (i) is Lemma \ref{avoid double tangent}, so we can find  loops in $(\PP^4)^\vee$ that induce nontrivial permutations of $\Gamma_1$ but act trivially on $\Gamma_2$ (and vice versa).

Now we prove (ii), which again is a small modification of the above cited proofs in the one-curve situation. We have $C = C_1 \cup C_2 \subset \PP^4$, where $C_i$ are nondegenerate, irreducible and reduced. 
To show that $\Gamma = C \cap H$ ($H$ general) is in LGP, it is enough to show that  $\Gamma$  does not contain four points in a plane, i.e., any four points of $\Gamma$ impose four conditions on linear forms. It is certainly true (by UPP) that no four points of $\Gamma$ coming from $C_1$ or from  $C_2$ are coplanar.
Let $a_1, a_2$ be integers with $a_1+a_2 = 4$, $a_1>0$, $a_2>0$. We will show that any $a_1$ points of $\Gamma_1$ together with any $a_2$ points of $\Gamma_2$ span the hyperplane $H$. We will give the case $a_1 = a_2 = 2$, but the other cases are analogous.

Consider the incidence locus
\[
I \subset ( \PP^4)^\vee \times C_1^2 \times C_2^2
\]
be defined by
\[
I = \{ (H, P_1 + Q_1, P_2 + Q_2 ) \ | \ P_1, P_2, Q_1, Q_2 \in H \}.
\]
Let $ J \subset I$ be defined by
\[
 J = \{ (H, P_1 + Q_1, P_2+Q_2) \in I \ | \ P_1, Q_1, P_2, Q_2 \hbox{ span a plane} \}.
\]
(If $P_i = Q_i$, $H$ is tangent to $C_i$ at that point.)

It follows from (i) that $I$ is irreducible. Note that    ${J} \subset I$ is closed, and that $J \neq I$ since you can find two points of $C_1$ and two points of $C_2$ such that the union does not span a plane. We have  $\dim I = 4$ and $\dim  J \leq 3$.

Let $\pi_1 \colon I \rightarrow (\PP^4)^\vee$ be the projection to the first factor. Then $\pi_1(J)$ omits an open set. Hence a general hyperplane section of $C$ is in LGP. This proves (ii).

Part (iii) follows from the fact that both $\Gamma_1$ by itself and $\Gamma_2$ by itself have the UPP, and from the uniqueness of $Y_1$ and $Y_2$.
\end{proof}

The following set of results will be very  useful for sets of points in LGP.

\begin{theorem}[Maroscia, {\cite[Theorem 2.3]{maroscia}}]\label{Maroscia}
Let $P_1,\dots,P_s$ be distinct points in $\PP^n$ ($s > n \geq 2$) with $h$-vector $(1,h_1,\dots,h_m)$, where $h_m > 0, h_{m+1} = 0$.  Suppose that 
\begin{itemize}
    \item[\em (a)] $h_{m-1} = n-h$ with $h \geq 1$, and
    \item[\em (b)] no $n-h+1$ of the $s$ points lie in a $\PP^{n-h-1}$ of $\PP^n$.
\end{itemize}
Then at least $(n-h)(m-1)+2$ of the $s$ points lie in a $\PP^{n-h}$ of $\PP^n$. In particular, $Z$ is not in LGP.

\end{theorem}

\begin{lemma}  [Castelnuovo's Lemma] \label{Cast} Let $Z$ be a set of $q$ points in LGP in $\PP^n$. Assume $q\geq 2n+3$ and assume that the Hilbert function of $Z$ satisfies $h_Z(2)\leq 2n+1$.  Then  in fact $h_Z(2) = 2n+1$ and $Z$ lies in a rational normal curve in $\PP^n$.
\end{lemma}

\begin{proof}
    If the $h$-vector of $Z$ is $(1,n, h_2,h_3, \dots)$, the assumption $h_Z(2) \leq 2n+1$ means $h_2 \leq n$ and the assumption $q \geq 2n+3$ implies $h_3 > 0$. If we had $h_2 < n$, Maroscia's theorem above implies that $Z$ is not in LGP. Hence we have $h_Z(2) = 2n+1$. The last part then follows from Lemma (3.9) of \cite{montreal}.
\end{proof}

For $Z \subset \PP^3$, Castelnuovo's Lemma becomes the  following.

\begin{corollary} \label{CL}

Let $Z \subset \mathbb P^3$ be a set of $\geq 9$ points in LGP, with $h$-vector $(1,3, a_2,\dots,a_m)$.
\begin{itemize}
 
\item[\em (a)]
 If $a_2 = 3$ then $Z$ lies on a twisted cubic curve. 

\item[\em (b)]
Without assuming $a_2=3$, if $a_i  = 3$  for some $3 \leq i \leq m-2$, we get that $a_j = 3$ for $i \leq j \leq m-1$,   $a_m \leq 3$, and at least $3(m-1) + a_m + 1$ of the points lie on a twisted cubic.
\end{itemize}

\end{corollary}

\begin{proof}
Part (a) is just a translation of the standard Castelnuovo's Lemma given above to $\PP^3$. For (b),
    assume that $a_i = 3$ for $3 \leq i \leq m-2$. By Macaulay's Theorem (see \cite[Theorem 4.2.10 c)]{BrunsHerzog}), $a_{i+1} \leq 3$. By  Theorem \ref{Maroscia} and LGP, $a_{i+1} > 2$. Thus $a_{i+1} = 3$, which must then be maximal growth of the Hilbert function $h$. The result follows from LGP and \cite{BGM}.
\end{proof}

\subsection{Lifting properties from the hyperplane section}\label{sec: lifting}

In this section we will consider curves $C \subset \mathbb P^4$ and address the following sort of question inspired by Laudal's Lemma. 

\begin{question}
If the general hyperplane section of $C$ lies on a curve $Y$ in $\mathbb P^3$ of a certain kind, when can we conclude that $C$ lies on a surface of degree $ (\deg Y)$ whose hyperplane section is $Y$? 
\end{question}

In part, we will  be looking for the smallest degree for $C$ that guarantees this kind of conclusion. The results will be used later in the paper.
Of course this is not automatic. For example, if $C$ consists of six general lines in $\mathbb P^4$ then the general hyperplane section of $C$ lies on a twisted cubic in $\PP^3$, while $C$ itself does not even lie on one quadric hypersurface, much less on a surface of minimal degree (which lies on three quadric hypersurfaces). In this section we give results in this  direction.

To set notation, let $R = \CC[x_0,\dots,x_n]$, and let $L$ be a general linear form. Let $S = R/(L) \cong \CC[x_0,\dots,x_{n-1}]$. Let $C$ be an equidimensional curve in $\mathbb P^n$ with no embedded points (we will add assumptions as needed) with saturated ideal $I_C$. Let $H$ be the hyperplane defined by $L$ and let $\Gamma = C \cap H \subset H \cong \mathbb P^{n-1}$.    
Since $\Gamma$ may be viewed as a set of points in $H$ or in $\mathbb P^n$, to emphasize that we are in the former case we will use the notation $I_{\Gamma|H}$. Note that $I_{\Gamma|H} \cong \left ( \frac{I_C + (L)}{(L)} \right )^{\textrm{sat}}$. The saturation is not necessary when $C$ is ACM.

The linear form $L$ induces a homomorphism 
\[
\times L \colon \bigoplus_{t \in \mathbb Z} H^1(\mathcal I_C (t)) (-1) \longrightarrow \bigoplus_{t \in \mathbb Z} H^1(\mathcal I_C(t)).
\]
The kernel, $K$, and cokernel, $W$, of this homomorphism are graded $R$-modules. We also denote by $J$ the module $I_C / (L \cdot I_C)$, which is isomorphic to the ideal  $\frac{I_C + (L)}{(L)}$ in $S$, and whose saturation is $I_{\Gamma|H}$.
Notice that $K$ (resp. $W$) is nonzero if and only if $C$ is not ACM if and only if $J$ is not saturated.

Let $t \in \mathbb Z$. Consider the commutative diagram
{\footnotesize \tiny
\begin{equation}\label{longseq}
\begin{array}{cccclccclccccccc}
0  \rightarrow  [I_C]_{t-1}  \stackrel{\times L}{\longrightarrow}  [I_C]_t  & \longrightarrow &  [I_{\Gamma|H}]_t & \longrightarrow & H^1(\mathcal I_C(t-1)) & \stackrel{\times L}{\longrightarrow} & H^1(\mathcal I_C(t)) & \longrightarrow & H^1(\mathcal I_{\Gamma|H}(t)) & \rightarrow & \dots. \\
\hfill \searrow &&  \nearrow \hfill \searrow && \nearrow \hfill & & \hfill \searrow && \nearrow \hfill  \\
& [J]_t && [K]_t &&&& [W]_t \\
\hfill \nearrow && \searrow \hfill \nearrow & & \searrow  \hfill &&  \hfill \nearrow && \searrow \hfill  \\
\hspace{.8in} 0 &&  0  & &  \hspace{.2in} 0 \hfill & & \hspace{.35in} 0 \hfill & & \hspace{.2in}  0. \hfill
\end{array}
\end{equation}
}

\begin{remark} \label{HU conclusion}
For a finitely generated graded module $M$ we denote by $\alpha(M)$ the initial degree of $M$. Let $A = S/I_{\Gamma|H}$ and $\bar A$ its artinian reduction. Let $a$ be the least degree of a socle element of $\bar A$ and let $b = \alpha(K)$. Note that $b$ is the least degree in which $I_{\Gamma|H}$ contains an element that does not lift to $I_C$.
Then \cite[Theorem 3.16]{HU} implies the following:  
\[
\hbox{
\begin{tabular}{p{5.5in}}
{\em
If $C$ is not ACM then $a \leq b$. That is, the least degree in which $I_{\Gamma|H}$  contains an element that does not lift to $I_C$ is greater than or equal to the least degree of a socle element of the artinian reduction of $S/I_{\Gamma|H}$.  }
\end{tabular}
}
\]
\end{remark}

For the next result, recall that $\Gamma$ is {\em level} if the socle of $\bar A$ is concentrated in the last degree.

\begin{corollary} \label{socle r}
Assume that $\Gamma$ is level with socle in degree $r$. Then for $t < r$, all of  $[I_{\Gamma|H}]_t$ lifts to $[I_C]_t$.
\end{corollary}

\begin{remark}
In the case of curves in $\mathbb P^3$, Remark \ref{HU conclusion} is essentially the content of a theorem of Strano, which is phrased a bit differently in \cite[Theorem~1]{strano}. He says that for a curve $C$ in $\mathbb P^3$ with general hyperplane section $\Gamma = H \cap C$,  if $\Tor^R_1 (I_{\Gamma|H})_n  =0$ for every $0 \leq n \leq \sigma+2$ then every curve in $H$ of degree $\sigma$ containing $\Gamma$ can be lifted to a surface of degree $\sigma$ containing $C$. This means that if 
\[
\sigma < \min \{ \hbox{degree of socle elements of } \bar A \}
\]
then every curve of degree $\sigma$ lifts, which is essentially Remark \ref{HU conclusion}.
Strano points out that this result can be used to prove Laudal's Lemma, namely that if $C$ is irreducible of degree $d > \sigma^2+1$ and $\Gamma$ is contained in a curve of degree $\sigma$ then $C$ is contained in a surface of degree $\sigma$. It suffices to show that under the stated assumptions, the socle of $\bar A$ must begin in degree $\geq \sigma + 1$.

One would like to find an analogue of Laudal's Lemma for curves in higher dimensional projective space. For example, if $C \subset \mathbb P^4$ is reduced and irreducible and $\deg C > N$ (for suitable $N$), and if a general hyperplane section $\Gamma$ of $C$ lies on a surface of degree $\sigma$ then $C$ lies on a hypersurface of degree $\sigma$. Even more interesting: if $\Gamma$ lies on a curve of degree $\sigma$, when can we conclude that $C$ lies on a surface of degree $\sigma$? Although we have less complete knowledge of the possible Hilbert functions of sets of points in $\mathbb P^3$ with the uniform position property, in this section we will give some results in this direction.
\end{remark}

\begin{lemma} \label{lift to surf 2}
    Let $C \subset \PP^n$ be a curve (not necessarily irreducible). Let $H$ be a general hyperplane and let $\Gamma = C \cap H$. Assume that there is an equidimensional curve $Y$ in $H$  and an integer $t$ so that
    \begin{itemize}
        \item[(i)] $[I_{\Gamma|H}]_t = [I_{Y|H}]_t$ and the base locus of $[I_{Y|H}]_t$ is $Y$;
\item[(ii)] the socle of the artinian reduction of $S/I_{\Gamma|H}$ begins in degree $> t$.     
    \end{itemize}
\noindent  Then $C$ lies on an equidimensional surface $T$  whose hyperplane section is $Y$.  The surface $T$ is ACM if and only if $Y$ is ACM.
\end{lemma}

\begin{proof}
    The assumption (ii) combined with the statement of Remark \ref{HU conclusion} gives us that the entire component $[I_{\Gamma|H}]_t$ lifts to $[I_C]_t$. But $[I_{\Gamma|H}]_t = [I_{Y|H}]_t$, where $Y$ is an equidimensional curve, and $H$ is a general hyperplane. This means that $[I_C]_t$ defines a surface of the same degree  as $Y$. Even if this surface is not equidimensional, one can take the top dimensional part to obtain $T$. The last statement of the lemma is standard since $Y$ is a general  hyperplane section of $T$ -- see for instance \cite[Proposition 2.1]{HU}.
\end{proof}

Next, we focus on the case in which $\Gamma$ lies in a curve of minimal degree, or in a complete intersection of two quadrics in $\mathbb P^3$. They indeed will prove to be the relevant cases for our analysis on the maximal intersection of $C_1$ and $C_2$.

Other cases will be irrelevant. For instance, $\Gamma$ could also lie on a rational quartic curve, so it would have $h$-vector $(1,3,5,4,4,\dots)$, and thus this case cannot give the best bound.

\begin{proposition} \label{rnc-nonacm}
Let $C \subset \mathbb P^n$ be a curve of degree $d\geq 2n-1$, which in addition we assume is ACM if $d=2n-1$. If a general hyperplane section $\Gamma$ of $C$ lies on a rational normal curve $Y$ in $H = \mathbb P^{n-1}$, then $C$ lies on a surface of minimal degree $n-1$ in $\mathbb P^n$. 
\end{proposition}

\begin{proof}
First assume $C$ is not ACM (and hence  $d\geq 2n$).
Since $\Gamma$ lies on a rational normal curve, the $h$-vector of $\Gamma$ has the form
\[
(1,n-1, n-1, \dots, n-1,c)
\]
where we can suppose that the last nonzero entry is in some degree $r$. Since $Y$ is a rational normal curve, the points of $\Gamma$ are in uniform position. Furthermore, the socle of the artinian reduction $\bar A$ of $S/I_{\Gamma|H}$ is concentrated in degree $r$ (i.e., $\bar A$ is level), and since $\deg C \geq 2n$, we get $r \geq 3$. Then by Corollary \ref{socle r}, $[I_{\Gamma|H}]_2 = [I_{Y|H}]_2$ lifts to $[I_C]_2$, and applying Lemma \ref{lift to surf 2} we get the desired conclusion.

Now we may assume $C$ is ACM with $d\geq 2n-1$ and $\Gamma$ lies in a rational normal curve in $H$. Let $L$ be a linear form defining our general hyperplane $H$. Then $I_{\Gamma|H} \cong \frac{I_C}{L \cdot I_C} \cong \frac{I_C + (L)}{(L)}$ so the lifting from $\Gamma$ to $C$ follows since $C$ and $\Gamma$ (in $H$) have the same Betti table, since $C$ is ACM.
\end{proof}

 The most important situation for us is introduced in the following result.

\begin{corollary} \label{lift ci22}
Let $C \subset \mathbb P^4$ be a nondegenerate curve and assume that a general hyperplane section $\Gamma$ of $C$ lies on a complete intersection curve $Y$ of type $(2,2)$ inside the hyperplane, and $\dim [I_{\Gamma|H}]_2~=~2$.

\begin{itemize}

    \item[(a)] If $C$ is ACM then $C$ lies on a complete intersection surface of type $(2,2)$.

\item[(b)] If $C$ is not ACM, assume that it is either a reduced, irreducible curve or the union $C = C_1 \cup C_2$ of two nondegenerate, reduced, irreducible curves,  and assume that  $\deg C \geq 10$.   Then $C$ lies on an irreducible complete intersection surface of type $(2,2)$.
\end{itemize}
\end{corollary}

\begin{proof}
Assertion (a) is standard since the general hyperplane section preserves the graded Betti numbers. Hence from now on we assume $C$ is not ACM.

Note that if $C$ is reduced and irreducible then $\Gamma$ has the Uniform Position Property (UPP) thanks to \cite{harris}, while if $C$ is the union of two nondegenerate reduced, irreducible curves then $\Gamma$ is at least in LGP, by Lemma \ref{lgp rmk}.  

We first claim that LGP forces $Y$ to be irreducible. Indeed, if $Y$ is reducible then it is either the union of a twisted cubic and a line, or else the union of two or more plane curves. The latter forces the general hyperplane section of $C$ to fail to be in LGP, so without loss of generality assume it is the former. Then Monodromy as in Lemma \ref{lgp rmk}  (iii) forces either $C_1$ or $C_2$ to be a plane curve (so that the general hyperplane section lies on a line), violating the assumption that they are both nondegenerate.

As a result, the $h$-vector of $\Gamma$ must have the value 4 in degree 2, and it must avoid maximal growth $(\dots, a, a, \dots)$ for any $a \leq 3$ by \cite{BGM}. 

 Next we take care of the cases $d = 10, 11, 12$. Thanks to Theorem \ref{Maroscia}, the corresponding $h$-vectors are $(1,3,4,2)$, $(1,3,4,3)$ and either $(1,3,4,4)$ or $(1,3,4,3,1)$. We begin with the first two of these.
Because $Y$ is irreducible, these two possibilities for $\Gamma$ both lie in a complete intersection of type $(2,2,3)$. Then the residual is a set of (respectively) 2 points or 1 point, and by a standard mapping cone argument one gets that the artinian reduction of $S/I_{\Gamma|H}$ has socle only in degree 3. Similarly, if the $h$-vector is $(1,3,4,3,1)$ then the irreducibility of $Y$ and the existence of an independent cubic generator forces $\Gamma$ to be a complete intersection, and then  \cite[Theorem 2.4]{Mig} forces $C$ to be a complete intersection, hence ACM.
For the last case $(1,3,4,4)$ one sees that the fact that $\Gamma$ lies on an irreducible complete intersection of type $(2,2)$ forces the socle to be entirely in degree 3. In all of these cases, $[I_{\Gamma|H}]_2$ lifts to $[I_C]_2$ by Remark \ref{HU conclusion} and we are done. We can now assume that $\deg C \geq 13$.
\medskip

Our goal now is to show that the socle must begin in degree $\geq 3$. Then by Remark \ref{HU conclusion}, all of $[I_\Gamma]_2$ lifts to $I_C$, so the complete intersection curve in $H$ containing $\Gamma$ lifts to a complete intersection surface containing $C$.

The assumption that $\deg C \geq 13$ and the observation above about avoiding maximal growth means, a priori, that the possible $h$-vectors are given in Table \ref{table:possible h}. 
\renewcommand{\arraystretch}{1.2}
\begin{table}[h!] 
\[
\begin{array}{cc|cccccccccccccc}
\multirow{2}{*}{Case}& &\multicolumn{8}{c}{\mbox{degree}} & \\
& & 0 & 1 & 2 & 3 & \dots & r-3 & r-2 & r-1 & r & r+1 \\
\hline
(1) & & 1 & 3 & 4 & 4 & \dots & 4 & 4 & 4 & 4 & 0 \\
(2) & & 1 & 3 & 4 & 4 & \dots & 4 & 4 & 4 & 3 & 0\\
(3) & & 1 & 3 & 4 & 4 & \dots & 4 & 4 & 4 & 2& 0 \\
(4) & & 1 & 3 & 4 & 4 & \dots & 4 & 4 & 4 & 1& 0 \\
(5) & & 1 & 3 & 4 & 4 & \dots & 4 & 4 & 3 & 2 & 0\\
(6) & & 1 & 3 & 4 & 4 & \dots & 4 & 4 & 3 & 1 & 0\\
(7) & & 1 & 3 & 4 & 4 & \dots & 4 & 4 & 2 & 1 & 0 \\
(8) & & 1 & 3 & 4 & 4 & \dots & 4 & 3 & 2 & 1 & 0 
\end{array}
\]
\caption{All possible $h$-vectors.}
\label{table:possible h}
\end{table}

We first eliminate Cases (5), (7) and (8). Since $\Gamma$ is in LGP, Theorem \ref{Maroscia} (Maroscia's theorem) eliminates (7) and (8). Consider (5). Since $Y$ is irreducible, $\Gamma$ lies in a complete intersection of type $(2,2,r-1)$. Such a complete intersection has $h$-vector $(1,3,4,4,\dots,4,3,1)$ that is smaller than (5) in degree $r$, so it cannot exist.

The remaining cases do occur and we claim the following  (remembering that $\deg C \geq 13$).

\renewcommand{\arraystretch}{1.2}
\begin{table}[h!]
\begin{center}
\begin{tabular}{c|llll}
Case & smallest socle degree \\ \hline
(1) & $r \geq 4$  \\
(2) & $r \geq 4$ \\
(3) & $r \geq 3$  \\
(4) & $r-1 \geq 3$ ( $\Rightarrow$ $r \geq 4$) \\
(6) & $r \geq 4$ \ \  (This is a complete intersection.)
\end{tabular}
\end{center}
\caption{Actually occurring cases of the $h$-vector.}
\end{table}

\noindent 

In Case (1) it is clear that the socle is 4-dimensional, occurring in degree $r \geq 4$.

In Case (2), as above  we can link with a complete intersection of type $(2,2,r)$ to a single point and the claim follows as it did with the case $\deg C = 11$. In fact the socle is 3-dimensional, occurring in degree $r \geq 4$.

In Case (3), as above we can link with a complete intersection of type $(2,2,r)$ to a set of two points, and the claim follows as it did with the case $\deg C = 10$.

Case (4) is not symmetric, so it cannot be Gorenstein. Thus there must be more socle than just degree $r$, and the only place it can be is $r-1$, since from degrees 2 to $r-1$ multiplication by a general linear form is, as for the hyperplane section of $Y$, an isomorphism.  But we have $\deg C \geq 13$, so the socle is in degree at least 3. This case does occur: it is the residual in a complete intersection of type $(2,2,r)$ of a set of 3 points.

For Case (6), again the irreducibility of $Y$ gives a complete intersection of type $(2,2,r-1)$ containing $\Gamma$ and having the same $h$-vector, hence $\Gamma$ is a complete intersection,  whose artinian reduction has socle only in degree $r$.

Thus we obtain that $[I_{\Gamma|H}]_2$ lifts to $I_C$  in all cases, thanks to Remark \ref{HU conclusion}, so we are done. 
\end{proof}

\begin{remark}
    We believe that Corollary  \ref{lift ci22} (b) holds also for curves of degree 9, but we have no proof. Notice that the $h$-vector of the general hyperplane section would be $(1,3,4,1)$, which forces socle in degree 2. However, the corollary does not hold for curves of degree 8; for example, a general rational curve of degree 8 lies on no quadric hypersurface (much less a complete intersection of type $(2,2)$) but its general hyperplane section lies on a pencil of quadrics.
\end{remark}

Finally, for future reference we state the following result, presented here in a form adapted to our purposes (recall that in this  paper ``curve" includes an assumption of reducedness):
\begin{lemma} [\cite{Mig}, Theorem 2.4] \label{mig fact}
Let $C$ be a nondegenerate curve in $\PP^4$. Let $Z$ be the scheme $H\cap C$, where $H$ is a general hyperplane. If $Z$ is a complete intersection in $H$, then $C \subset \PP^4$ is a complete intersection.
\end{lemma}

\subsection{Restrictions on the $h$-vector}
\label{sec: restrictions on h-vector}

In this section we narrow  down the possible $h$-vectors of $\Gamma = C \cap H = (C_1 \cup C_2) \cap H$ (where $H$ is a general hyperplane) that are candidates for yielding the largest genus for $C = C_1 \cup C_2$ (Proposition \ref{min hvtr} and subsequent lemmas). 

Giuffrida result \eqref{eq: G bound in Pn} applied to reduced, irreducible, nondegenerate curves $C_1, C_2$ in $\PP^4$ (without assuming that they have the same degree  or lie on surfaces of any specific type) gives 
$$|C_1\cap C_2|\leq (d_1-2)(d_2-2) + 2.$$

Assuming without restriction that $d_1\leq d_2$, our Theorem \ref{bd on cubic} and Remark \ref{sing cubic} give a new bound
$$|C_1\cap C_2|\leq \frac{1}{2}d_1(d_2-1),$$
which improves the above bound roughly by the factor of $1/2$, see in particular Table \ref{tab: compare G and POLITUS}. We also know that our bound (stated more carefully in Theorem \ref{bd on cubic}) is sharp, for any $d_1$ and $d_2$, among curves lying on a cubic surface. Now we consider curves not lying on a cubic surface.

\begin{remark} \label{rem:lowdeg}
    The purpose of Proposition \ref{min hvtr}  is to limit the possibilities for the $h$-vector of $C \cap H$, and for us the interesting situation is $\deg C \geq 9$. However, we remark that ignoring the issue of whether $C$ lies on a cubic surface, we also know the possible $h$-vectors when $d \leq 8$ assuming that $C$ is either reduced, irreducible and nondegenerate or the union $C = C_1 \cup C_2$ of two such curves. These $h$-vectors are
    \[
(1,3,4) , \
(1,3,3,1) ,  \
(1,3,3) , \
(1,3,2) , \
(1,3,1) , \
(1,3).
    \]
    Indeed, nondegeneracy gives us $(1,3,\ldots)$ and Maroscia's theorem (Theorem \ref{Maroscia}) excludes the other possibilities. One additional exclusion will be given in Lemma \ref{elim (3,2)}.
    
\end{remark}

\begin{proposition} \label{min hvtr}
Let $C $ be a curve  of degree $\geq 9$ in $\mathbb P^4$ that is either reduced, irreducible and nondegenerate, or the union $C = C_1 \cup C_2$ of two such curves (not necessarily of the same degree). 

\begin{itemize}
    \item If $C$ is irreducible, assume that it does not lie on a surface of minimal degree 3;
    \item if $C = C_1 \cup C_2$, assume that neither $C_1$ nor $C_2$ lies on a surface of minimal degree. 
\end{itemize}

\noindent We do not assume that any  of these curves is ACM.  

Let $\Gamma$ be a general hyperplane section of $C$. 
Assume that the $h$-vector of $\Gamma$ has the form
\[
(1,a_1,a_2,\dots,a_{r-1}, a_{r} )
\]
where $a_r \geq 1$, and the $h$-vector is zero beyond degree $r$. Then we have 

\begin{itemize}
    \item[(a)] $a_1=3$ and $a_2\geq 4$;
    \item[(b)] the following possibilities hold:

\begin{itemize}
\item[\em (i)] $r=2$, $5 \leq a_2 \leq 6$;   

\item[\em (ii)] $r = 3$, $a_2 = 4$, $1 \leq a_3 \leq 4$; 

\item[\em (iii)] $r=3$, $a_2 = 5$, $1 \leq a_3 \leq 7$;

\item[\em (iv)] $r=3$, $a_2 = 6$, $1 \leq a_3 \leq 10$;

\item[\em (v)] $r \geq 4$, $a_k \geq 4$ for $2 \leq k \leq r-2$, and $a_{r-1} \geq 3$.
    
\end{itemize}
\end{itemize} 
\end{proposition}

\begin{proof}

Note that in  part (b)   we do not assert that the five listed cases do occur, but rather that any other possibility does not occur.
We denote by $H$ the general hyperplane cutting out $\Gamma$. By Lemma \ref{lgp rmk}, $\Gamma$ is in LGP.

\vspace{.1in}

\noindent \underline{Claim 1}: $a_1 = 3$.

\vspace{.1in}

It is a standard fact that $\Gamma$ is nondegenerate in $H = \mathbb P^3$. For instance, since $C$ is nondegenerate  in $\PP^4$, we can take four points on $C$ spanning a hyperplane $H_0$. Since $H_0 \cap C$ is nondegenerate  in $H_0$, then also the general hyperplane section $\Gamma = C \cap H$ is as well. Thus $a_1 = 3$.

\vspace{.1in}

\noindent \underline{Claim 2}: $a_2 \geq 4$. 

\vspace{.1in}

Suppose first that $a_2 \leq 2$. In order to reach a degree of 9 or more, and since by  Macaulay's theorem (see \cite[Theorem 4.2.10 c)]{BrunsHerzog}) all $a_i$ must be $\leq 2$ for $i \geq 3$, we must have $2 \geq a_{r-1} \geq a_r > 0$. Then by Maroscia's \Cref{Maroscia}, we have a contradiction to the LGP property. Thus $a_2 \geq 3$.
\medskip 

Suppose now that $a_2 = 3$. 

 \begin{itemize}
 
\item  If $r=2$, then the  fact that the sum of the entries of the $h$-vector is equal to the degree means $\deg C = \deg \Gamma = 7$, contradicting the assumption that $\deg C \geq 9$. The same argument excludes $(1,3,4)$ so we have (i).
\medskip

\item Assume $r \geq 3$. By Corollary \ref{CL}, $\Gamma$ lies on a twisted cubic curve. The $h$-vector of $\Gamma$ is then
\[
(1,3,3,3, \dots, 3, a_r)
\]
where $1 \leq a_r \leq 3$ and $r \geq 3$.  By Proposition \ref{rnc-nonacm}, $C$ then lies on a surface of minimal degree 3, a contradiction.

\noindent
Therefore we have $a_2 \geq 4$.
\end{itemize}

\noindent This finishes Claim 2. 
  
Now that we know $a_2 \geq 4$, and we have dealt with the case $r=2$ above, we next handle the case $r=3$. We consider the various possibilities for $a_2$ and the consequences for $a_3$.

\begin{itemize}
\item Let $a_2=4$. 
Then Macaulay's theorem gives us $a_3 \leq 5$. Since $r=3$ we get $a_3 \geq 1$ by definition of $r$.  Suppose $a_3 = 5$. Then the growth from degree 2 to degree 3 is maximal according to Macaulay. In fact, this is exactly the content of \cite[Example 2.11]{BGM}, in which it is shown that $\Gamma$ wildly fails to be in LGP (it forces 3 or 4 points on a line and 10 or 9 points on a plane not containing that line). This proves that if $r=3$ and $a_2 = 4$ then $1 \leq a_3 \leq 4$, so we have possibility (ii).

  \vspace{.1in}

\item Let $a_2=5$. Then the statement  $a_3 \leq 7$  is simply Macaulay's Theorem, and the statement $1 \leq a_3$ is the assumption $r=3$. This gives (iii).

  \vspace{.1in}

\item Let $a_2=6$. Then  there is no condition on $a_3$ other than being nonzero, which is the statement of (iv).

\end{itemize}
 
\vspace{.1in}

 So we can now assume $a_2 \geq 4$ and  $r \geq 4$.
Recall that  $ \sum_{i=0}^r a_i = d \geq 9$.

\vspace{.1in}

Assume $r=4$. We only have to show that $a_3 \geq 3$. But if $a_3 \leq 2$, Macaulay's Theorem gives $a_4 \leq 2$ as well, and then we get a violation of LGP thanks to Maroscia's Theorem \ref{Maroscia}. Thus the case $r=4$ is done, and we can assume $r \geq 5$.

\vspace{.1in}

\noindent \underline{Claim 3}: {\it If $r \geq 5$ then $a_i \geq 4$ for $3 \leq i \leq r-2$}.

\vspace{.1in}

We have two possibilities.

\begin{itemize}

\item  $a_i < 3$ for some $i$ in the  range $3 \leq i \leq r-1$. (Note that here we include a statement about $a_{r-1}$ which is not part of this claim but will be used in Claim 4.)

Then we have an immediate contradiction to LGP by Macaulay's theorem and Maroscia's theorem (Theorem \ref{Maroscia}). Indeed, since $i \geq 3$, we get $3 > a_i \geq a_{i+1} \geq a_r > 0$ by Macaulay's theorem, and then a contradiction to LGP by Maroscia's theorem. 

\item  $a_i = 3$ for some $i$ with $3 \leq i \leq r-2$. 

By Corollary \ref{CL} we have at least $3(r-1) + a_r + 1$ of the points on a rational normal curve. 

If $C$ is irreducible then by the Uniform Position Property of Harris \cite{harris} we get that all the points lie on the same rational normal curve in $H = \PP^3$. Then by Proposition \ref{rnc-nonacm}, $C$ lies on a surface of minimal degree, giving a contradiction.

So assume $C = C_1 \cup C_2$ with general hyperplane section $\Gamma = \Gamma_1 \cup \Gamma_2$. We have at least $3(r-1) + a_r +1$ of the points of $\Gamma$ lying on a rational normal curve in $H$ (a twisted cubic). If this number is at least 15 then we have at least 8 points of $\Gamma_1$ or of $\Gamma_2$ on a rational normal curve in $H$, so by Proposition \ref{rnc-nonacm} either $C_1$ or $C_2$ lies on a cubic surface, again giving a contradiction.

The same conclusion also holds (again by  Proposition \ref{rnc-nonacm}) if both curves have degree 7 and at least one is ACM.

The only remaining situation is where  $m=5$, $a_r = 1$ and $3(r-1)+a_r+1 = 14$. Then the $h$-vector of $C$ is $(1,3,3,3,3,1)$. By Castelnuovo's Lemma (Lemma \ref{Cast}), these 14 points lie on a twisted cubic curve. Then even though $C$ is reducible, the artinian reduction of $\Gamma$ clearly has no socle in degree $\leq 3$ so Remark \ref{HU conclusion} shows that $C$ lies on a surface of minimal degree, again a contradiction.
\end{itemize}

This concludes the proof of Claim 3.

\vspace{.1in}

\noindent \underline{Claim 4}: $a_{r-1} \geq 3$.

\vspace{.1in}

It follows immediately from Maroscia's theorem and LGP that $a_{r-1} < 3$ is impossible.  
\end{proof}

\begin{remark}
It is possible to have $a_{r-1} = 3$ and $a_r =1$. For example, this is true if $\Gamma $ is arithmetically Gorenstein.
\end{remark}

\begin{lemma} \label{elim (3,2)}
    Let $\Gamma \subset \PP^3$ be a nondegenerate set of $\geq 13$ points in LGP not lying on a twisted cubic curve. Then the $h$-vector of $\Gamma$ cannot have the form
    \[
    (1,3,\underbrace{4,4,\dots,4}_m ,3,2).
    \]
\end{lemma}

\begin{proof}
    There is a pencil of quadrics containing $\Gamma$. The LGP assumption precludes that the base locus of this pencil is a plane, and it precludes that the base locus is reducible (since there are no three points on a line or four in a plane). The ideal consists only of these two quadrics until degree $m+2$, and then the irreducibility of the base locus of the quadrics guarantees that $I$ contains a regular sequence of type $(2,2,m+2)$. But the quotient of such a regular sequence has $h$-vector with value 1 in degree $m+3$, which  is impossible since the $h$-vector of $\Gamma$ has value 2 in degree $m+3$.
\end{proof}

\begin{remark} \label{a,b}
We make a technical calculation, whose relevance will be given in  Remark \ref{summary}. We fix $d$ to be the degree of $C_1 \cup C_2$.

Consider $h$-vectors of the form 
\[
(1,3,\underbrace{4,\dots,4}_m,a,b) 
\]
(where we  allow $b=0$). Let $p$ be the element of $\{1,2,3,4\}$ that is congruent to $d$ mod 4. Then to maximize $g(\underline{h})$ we want the $h$-vector
\[
\left \{
\begin{array}{ll}
(1,3,\underbrace{4,4,\dots, 4}_m , p) & \hbox{if $p \neq 4$ (so  $(a,b) = (p,0)$)}, \\
(1,3,\underbrace{4,4,\dots,4}_{m} , 3,1) & \hbox{if }p=4 \hbox{ (so } (a,b) = (3,1)).
\end{array}
\right.
\]
(We choose $(1,3,\underbrace{4,4,\dots,4}_{m} , 3,1)$  over $(1,3,\underbrace{4,4,\dots, 4}_{m+1} )$ thanks to Corollary~\ref{hvtr}, and we are forced to take $(1,3,\underbrace{4,4,\dots, 4}_m , 1)$ over $(1,3,\underbrace{4,4,\dots,4}_{m-1} , 3,2)$ because of Lemma \ref{elim (3,2)}.)

For example say $d = 20$. Then the $h$-vectors
\[
(1,3,4,4,4,4) \ \ \ \hbox{ and } \ \ \  (1,3,4,4,4,3,1)
\]
are the only possibilities. An easy calculation (see Corollary \ref{hvtr}) shows the second gives the bigger genus.

\end{remark}

\begin{remark} \label{summary}
In this remark we summarize our strategy based on the technical preparatory results given so far. We maintain the notation introduced in this section.

Given irreducible curves $C_1$ and $C_2$, we consider the union $C = C_1 \cup C_2$. Let $\Gamma$ be the general hyperplane section of $C$ and let $\underline{h}$ be the $h$-vector of $\Gamma$, i.e., the first difference of the Hilbert function $h_\Gamma(t)$. We have seen using  Proposition \ref{rosa fact} and \ref{genus formula} and Remark \ref{formal calc} that
\[
|C_1 \cap C_2| \leq  g - g_1 - g_2 +1 \leq g+1 \leq  
\sum_{i=1}^\ell [d - h_\Gamma(i)]  +1 = g(\underline{h}) +1.
\]

Unfortunately we are given only $d_1, d_2$ and consequently their sum $d$. To find an upper bound for $|C_1 \cap C_2|$ (before comparing it to the target $B(d_1,d_2)$) we need to bound $g(\underline{h})$ as $\underline{h}$ varies over all possible $h$-vectors $\underline{h}$ for $\Gamma$. One goal of this section has been to narrow down the possible $\underline{h}$ given our assumptions on $C_1$ and $C_2$. This required extensive preparation. 

The next step has been to find, among all remaining allowable $h$-vectors, the $h$-vector $\underline{h}$ that yields the largest value of $g(\underline{h})$. Using Proposition \ref{min hvtr} and Lemma \ref{elim (3,2)} (see Corollary \ref{hvtr}) and Remark \ref{a,b}, we showed that such an $h$-vector has the form 
\[
(1,3,\underbrace{4,\dots,4}_m,a,b) 
\]
(where we abuse notation and allow $b=0$). Then depending on congruence of $d\pmod{4}$, the biggest theoretical genus for a curve not lying on a cubic surface will come from an $h$-vector ending 
\[
(4,3,1),\; (4,1),\; (4,2) \mbox{ or } (4,3).
\]

\end{remark}

\begin{lemma} \label{genus from hvector1}

Let $C$ be a curve in $\mathbb P^4$ that is either reduced, irreducible, nondegenerate and not lying on a cubic surface, or the union $C_1 \cup C_2$ of two such  curves (not necessarily of the same degree). Let $\Gamma$ be a  general hyperplane section. Suppose the $h$-vector of $\Gamma$ is
\[
(1,3,\underbrace{4,4,\dots, 4}_m , a,b),
\]
where $(a,b)$ are as in Remark \ref{a,b}.  Let $g$ be the arithmetic genus of $C$. Then
\[
g \leq (m+2)b + (m+1)a + 4 \binom{m+1}{2}.
\]
Equality holds if and only if $C$ is ACM.
 If $C$ is ACM and $m \geq 1$ (i.e., $\deg C > 8$) then $C$ lies on an irreducible complete intersection quartic surface. If $C$ is not ACM we assume $\deg C \geq 10$, and again $C$ lies on an irreducible complete intersection quartic surface.
\end{lemma}

\begin{proof}
We compute the genus assuming that $C$ is ACM. If it is not, then the genus is lower (see Proposition \ref{genus formula}). From the $h$-vector we get the Hilbert function of the general hyperplane section to be 
\[
(1,4,2(4), 3(4), 4(4), \dots, m(4), (m+1)4, (m+1)4 +a, (m+1)4 + a + b, (m+1)4 + a + b, \dots).
\]
Then the above references give
\[
g \leq (b) + (a+b) + (a+b+4) + (a+b+4(2)) + \dots + (a+b+4(m-1)) + (a+b+4m)
\]
which is equal to the claimed bound, with equality  if and only if $C$ is ACM.

 The last claim follows immediately from Corollary \ref{lift ci22} (a) in the ACM case and Corollary~\ref{lift ci22}(b) in the non-ACM case.
\end{proof}

\begin{remark} \label{lowest h-vector}
The values of the $h$-vector given in the statement of Lemma \ref{genus from hvector1} are term by term the lowest possible given the degree of $C$ and given the fact that $C$ does not lie on a surface of degree 3, and that the general hyperplane section has LGP (Lemma \ref{lgp rmk}). This optimal $h$-vector gives an upper bound for the genus, by Corollary \ref{hvtr}. By Proposition \ref{genus formula}, the maximum genus among curves whose hyperplane section has this $h$-vector comes when the curve is ACM.
\end{remark}

\begin{corollary}\label{cor: genus bound explicit}
Let $g(d)$ be the value obtained from the genus formula (see Remark \ref{formal calc}) applied to the ``extremal" $h$-vector described in Lemma \ref{genus from hvector1} (see also Remark \ref{lowest h-vector}).
And let $\ell$ be the remainder of the division of $d$ by 4. Then we have an explicit formula: 
$$g(d)=\left\{
\begin{array}{ccc}
    \frac{d^2-4d+8}{8} && \ell=0; \\
    \frac{d^2-4d+3}{8}  && \ell =1 \mbox{ or } \ell=3;\\
    \frac{d^2-4d+4}{8} && \ell=2.
\end{array}
\right.
$$
\end{corollary}
\begin{proof}
Indeed, taking the upper bound from Lemma \ref{genus from hvector1} as 
$$g(d)=(m+2)b+(m+1)a+2m(m+1)$$
we have in the case of the $h$-vector ending $(4,3,1)$ (so $a=3$ and $b=1$) that $m=(d/4)-2$ and the formula follows.

The remaining cases work similarly.    
\end{proof}

\section{The genus bound and the situation when $|d_1 - d_2|$ is small} \label{sec: Bg leq B}

Let $C_1,C_2$ be reduced, irreducible curves in $\PP^4$. Let $g_1, g_2, g$ be the arithmetic genera of $C_1, C_2$ and $C = C_1 \cup C_2$. We know from Proposition \ref{rosa fact} that
\[
|C_1 \cap C_2| \leq g - g_1 - g_2 +1.
\]
If we seek the largest possible value of $|C_1 \cap C_2|$, we can let $g_1 = g_2 = 0$. Then in terms of $g$ this gives an upper bound for $|C_1 \cap C_2|$. Now to find the largest (theoretically) possible value of $|C_1 \cap C_2|$ we use the upper bound $g(d)$ on $g$ worked out in Lemma \ref{genus from hvector1}.

{
\begin{notation}\label{not: genus bound Bg}
    For positive integers $d_1, d_2$ we define the {\it genus bound} $B_g(d_1,d_2)$ by
    \[
    B_g(d_1,d_2) = g(d_1 + d_2)+1.
    \]
\end{notation}
}

\begin{remark}\label{rem: max bg implies acm}
A consequence of the introduced notation is that if
$$|C_1\cap C_2|=B_g(d_1,d_2),$$
then both curves are rational and their union is ACM.
\end{remark}

Conjecture \ref{conjecture b} predicts $B(d_1,d_2)$ as an upper bound for $|C_1 \cap C_2|$. We want to see to what extent the use of $B_g(d_1,d_2)$ helps us prove this. We note that if $B_g(d_1,d_2) \leq B(d_1,d_2)$ then our goal is reached. Unfortunately this is not always the case, and we first make two observations about this situation.

\begin{remark} \label{d1,d2 far apart}
    Without other considerations, the use of $B_g (d_1,d_2)$ has limitations. Indeed, an obvious bound for $|C_1 \cap C_2|$ is $d_1 d_2$, since the number of intersections can only go up if we project to $\PP^2$, and in $\PP^2$, $d_1 d_2$ is obviously the upper bound. However, if $d_1$ and $d_2$ are far apart, $B_g(d_1,d_2)$ can even exceed $d_1 d_2$. For instance, if $d_1 = 30$ and $d_2 = 450$, $d_1d_2 = 13,\hspace{-.02in}500$ but $B_g(d_1,d_2) = 28,\hspace{-.02in}562$ (more than double). In this section we study what happens if $d_1$ and $d_2$ are close together, and later we will study what happens if we make other assumptions about the curves. 
\end{remark}

\begin{remark}
Before proving one of our main results (Theorem \ref{B minus Bg}), but after our detailed analysis of the $h$-vectors, it is useful to step back and summarize our situation.

Our goal is to show that $B(d_1,d_2)$ is an upper bound for $|C_1 \cap C_2|$ in all cases, when both curves are reduced and irreducible. We first showed that it is true when both are on a cubic surface, or even when only one is on a cubic surface, so now we assume neither is on a cubic surface. Although in Section \ref{sec: ACM curves} we will use a completely different method to prove the result for ACM curves, for the most part our approach is the following.

We know that $|C_1 \cap C_2| \leq B_g(d_1,d_2)$ is always true, and when it happens that $B_g(d_1,d_2) \leq B(d_1,d_2)$, we are done. Unfortunately, it can happen that $B_g(d_1,d_2) > B(d_1,d_2)$ -- see for instance Figure \ref{fig:Bg vs B}. This  happens because the bound $B_g(d_1,d_2)$ is an offshoot of Propositions \ref{rosa fact} and \ref{genus formula}, by making the extreme assumptions that both $C_1$ and $C_2$ have arithmetic genus 0 and that $C = C_1 \cup C_2$ is ACM (so $k=0$ in Proposition \ref{genus formula}). 

We believe that given $d_1$ and $d_2$, in all cases that $B_g(d_1,d_2) > B(d_1,d_2)$, the above extreme assumptions cannot happen. However, as Theorem \ref{B minus Bg} and Figure \ref{fig:Bg vs B} illustrate, the reverse inequality often does occur, giving our conjecture in many cases.
    
\end{remark}

In this section we will always assume $d_1 \leq d_2$ and give conditions in terms of $d_2 - d_1$ (and sometimes also $d_1, d_2$ themselves) that guarantee $B_g(d_1,d_2) \leq B(d_1d_2)$. 

We introduce a numerical invariant $M(d_1,d_2)$ with values as in the cases explained in the following table. The order of the entries is based on the approach given in the proof and is not self-evident at this point.

\renewcommand{\arraystretch}{1.3}
\begin{table}[h!] 
\centering
  \begin{tabular}{c|c|c|c|c|c|c|c}
    \diagbox[]{$\alpha=d_1$ mod 4}{$\beta=(d_2-d_1)$ mod 4} & 0 & 1 & 2 & 3  \\
    \hline
    \multirow{2}{*}{0}     & Case I & Case V & Case IX & Case XIII \\
         & $4d_2 - 16$ & $4d_2 - 11$   & $4d_2 - 12$ & $4d_2 - 11$ \\ 
    \hline
    \multirow{2}{*}{1}     & Case X & Case XIV & Case II & Case VI  \\
          & 0 & $4d_1-11$ & $-1$ & $4d_1-11$  \\
    \hline
    \multirow{2}{*}{2}     & Case III & Case VII & Case XI & Case XV \\
    &$4d_2 - 16$&$4d_2 - 11$&$4d_2 - 12$&$4d_2 - 11$\\
    \hline
    \multirow{2}{*}{3}     & Case XII & Case XVI & Case IV & Case VIII \\
      &$0$&$4d_1 - 11$&$- 1$&$4d_1 - 11$
    \end{tabular}
    \bigskip
    \caption{The value of $M(d_1,d_2)$ for Cases I--XVI in Theorem \ref{B minus Bg}}
    \label{conj cases}
\end{table}

\begin{theorem} \label{B minus Bg}
    Let $C_1, C_2$ be reduced, irreducible, nondegenerate curves in $\PP^4$, neither lying on a cubic surface. Assume that $(d_2-d_1)^2 \leq M(d_1,d_2)$ where $M(d_1,d_2)$ is as specified by Table \ref{conj cases}. Then $B_g(d_1,d_2) \leq B(d_1,d_2)$, so for these degrees Conjecture \ref{conjecture b} is true, and then the curves achieving the maximum $B_g(d_1,d_2)$ lie on a quartic surface. Furthermore, if $(d_2-d_1)^2 < M(d_1,d_2)$
    then the number of intersection points is strictly less than $B(d_1,d_2)$.

\end{theorem}

\begin{proof}
The statement about curves lying on a quartic surface comes from our results on lifting.

We set $\alpha$ and $\beta$ to be the values on the left and the top of Table \ref{conj cases}, i.e., 
\[
d_1 = 4u + \alpha \ \ (0 \leq \alpha \leq 3)
\]
and
\[
d_2 - d_1 = 4k+ \beta \ \ (0 \leq \beta \leq 3).
\]
Then we have
\[
d_2 = 4u+4k+\alpha + \beta
\]
and
\[
d = d_1 + d_2 = 8u + 4k + 2\alpha + \beta.
\]

Now, we consider sixteen cases depending on $\alpha$ and $\beta$. The values of $B_g$ and $B$ are computed using \eqref{eq: genus bound} and \eqref{eq: our bound} respectively.

\begin{multicols}{2}
\noindent \underline{Case I}:

$(\alpha = \beta = 0$, $\ell=0$)

$B_g = 8u^2 + 8uk - 4u - 2k + 2k^2 +2$

$B =  8u^2 + 8uk - 2u$

$B - B_g = 2u+2k-2k^2-2$

\vspace{.1in}


\noindent \underline{Case II}:

$(\alpha = 1, \beta=2$, $\ell=0$)

$B_g = 8u^2 + 8uk + 4u + 2k + 2k^2 +1 $

$B = 8u^2 + 8uk + 4u+1$

$B - B_g = -2k-2k^2-1$

\vspace{.1in}


\noindent \underline{Case III}:

($\alpha = 2, \beta = 0$, $\ell=0$)

$B_g = 8u^2 + 8uk + 4u + 2k + 2k^2 +2 $

$B = 8u^2 + 8uk + 6u + 4k+1 $

$B - B_g = 2u + 2k - 2k^2 - 1$

\vspace{.1in}


\noindent \underline{Case IV}:

($\alpha = 3, \beta = 2$, $\ell=0$)

$B_g = 8u^2 + 8uk + 2k^2 + 12u + 6k+6$

$B = 8u^2 + 8uk + 12u + 4k + 5$

$B - B_g = -2k-2k^2-1$

\vspace{.1in}


\noindent \underline{Case V}:

($\alpha = 0, \beta = 1$, $\ell=1$)

$B_g = 8u^2 + 8uk + 2k^2 -2u - k +1$

$B = 8u^2 + 8uk $

$B - B_g = -2k^2 + 2u + k -1$

\vspace{.1in}


\noindent \underline{Case VI}:

($\alpha = 1, \beta = 3$, $\ell=1$)

$B_g = 8u^2 + 8uk + 2k^2 + 6u + 3k +2$

$B = 8u^2 + 8uk + 8u$

$B - B_g =2u - 3k - 2k^2 -2$

\vspace{.2in}


\noindent \underline{Case VII}:

($\alpha = 2, \beta = 1$, $\ell=1$)

$B_g = 8u^2 + 8uk + 2k^2 + 6u + 3k +2$

$B = 8u^2 + 8uk + 8u + 4k + 2$

$B - B_g = 2u + k - 2k^2 $

\vspace{.1in}


\noindent \underline{Case VIII}:

($\alpha = 3, \beta = 3$, $\ell=1$)

$B_g = 8u^2 + 8uk + 2k^2 + 14u + 7k +7$

$B = 8u^2 + 8uk + 16u + 4k + 6$

$B - B_g = 2u - 3k - 2k^2 -1$

\vspace{.1in}


\noindent \underline{Case IX}:

($\alpha = 0, \beta = 2$, $\ell=2$)

$B_g = 8u^2 + 8uk + 2k^2 +1$

$B = 8u^2 + 8uk + 2u$

$B - B_g = 2u - 2k^2 -1$

\vspace{.1in}


\noindent \underline{Case X}:

($\alpha = 1, \beta = 0$, $\ell=2$)

$B_g = 8u^2 + 8uk + 2k^2 +1$

$B = 8u^2 + 8uk + 1$

$B - B_g = -2k^2$

\vspace{.1in}


\noindent \underline{Case XI}:

($\alpha = 2, \beta = 2$, $\ell=2$)

$B_g = 8u^2 + 8uk + 2k^2 + 8u + 4k + 3$

$B = 8u^2 + 8uk + 10u + 4k + 3$

$B - B_g = 2u - 2k^2$

\vspace{.1in}


\noindent \underline{Case XII}:

($\alpha = 3, \beta = 0$, $\ell=2$)

$B_g = 8u^2 + 8uk + 2k^2 + 8u + 4k + 3$

$B = 8u^2 + 8uk + 8u + 4k + 3$

$B - B_g = -2k^2$

\vspace{.2in}


\noindent \underline{Case XIII}:

($\alpha = 0, \beta = 3$, $\ell=3$)

$B_g = 8u^2 + 8uk + 2k^2 + 2u + k + 1$

$B = 8u^2 + 8uk + 4u$

$B - B_g = 2u - 2k^2 - k - 1$

\vspace{.1in}


\noindent \underline{Case XIV}:

($\alpha = 1, \beta = 1$,  $\ell=3$)

$B_g = 8u^2 + 8uk + 2k^2 + 2u + k + 1$

$B = 8u^2 + 8uk + 4u$

$B - B_g = 2u - 2k^2 - k - 1$

\vspace{.1in}


\noindent \underline{Case XV}:

($\alpha = 2, \beta = 3$, $\ell=3$)

$B_g = 8u^2 + 8uk + 2k^2 + 10u + 5k + 4$

$B = 8u^2 + 8uk + 12u + 4k + 4$

$B - B_g = 2u - 2k^2 - k$

\vspace{.1in}


\noindent \underline{Case XVI}:

($\alpha = 3, \beta = 1$, $\ell=3$)

$B_g = 8u^2 + 8uk + 2k^2 + 10u + 5k + 4$

$B = 8u^2 + 8uk + 12u + 4k + 4$

$B - B_g = 2u - 2k^2 - k$
\end{multicols}
\vspace{.1in}


\noindent The conditions in Table \ref{conj cases} are then a simple calculation, replacing $u$ and $k$ by the corresponding expressions in $d_1$ and $d_2$.
\end{proof}

In the case $d_1 = d_2$, we obtain a refinement of Theorem \ref{B minus Bg}, which allows us to establish \Cref{conjecture b} in full.

\begin{proposition}\label{t: bd in P4 same deg}
Let $C_1, C_2$ be nondegenerate, irreducible curves in $\mathbb P^4$ of the same degree $d \geq 6$. Then the number of intersection points of $C_1$ and $C_2$ is bounded by
\[
|C_1 \cap C_2| \leq B(d,d) = 
\left \{
\begin{array}{ll}
\frac{d(d-1)}{2} & \hbox{if $d$ is even}; \\
\frac{(d-1)^2}{2} + 1 & \hbox{if $d$ is odd.}
\end{array}
\right.
\]
Furthermore, suppose that $C_1$ and $C_2$ achieve this bound.

\vspace{.1in}

\begin{itemize}

\item If $d$ is even then $C_1$ and $C_2$ must both lie on the same surface of degree 3 (i.e., the bound is not achieved if $C_1, C_2$ fail to lie on a common cubic surface).

\vspace{.1in}

\item If $d$ is odd, then either

\vspace{.1in}

\begin{itemize}

\item $C_1$ and $C_2$ both lie on a common surface of degree 3 as described in Theorem ~\ref{bd on cubic}, or

\vspace{.1in}

\item $C_1$ and $C_2$ are curves of arithmetic genus 0, $C_1 \cup C_2$ is ACM, neither lies on a cubic surface, and $C_1 \cup C_2$ lies on a complete intersection surface of degree 4.

\end{itemize}

\vspace{.1in}

\noindent Both of these exist for any odd $d$.

\end{itemize}
\end{proposition}

\begin{proof}

Note that in this theorem $d$ represents the common degree, not the sum of the degrees. Since $d_1=d_2$, we have always $\beta=0$ and thus in Theorem \ref{B minus Bg}, Case I covers the case where $d \equiv 0 \hbox{ (mod 4)}$, Case X covers $d \equiv 1 \hbox{ (mod 4)}$, Case III covers $d \equiv 2 \hbox{ (mod 4)}$, and Case XII covers $d \equiv 3 \hbox{ (mod 4)}$. So the issue is to show the ``furthermore'' part of the Proposition.
\medskip

If $d$ is even, both  Case I and Case III of Theorem \ref{B minus Bg} give $B_g(d_1,d_2) < B(d_1,d_2)$ since $d_2 \geq 5$,  and we are done. 
\medskip

If $d$ is odd, Case X and Case XII give us $B_g \leq B$, and \Cref{lem: curves on del Pezzo quartic} (in the special case where the degrees are equal) gives the existence of curves achieving the bound.   It remains to show that neither $C_1$ nor $C_2$ is contained in a cubic surface. Consider $C_1$, without loss of generality. If $C_1$ lies on both a cubic surface and a quartic complete intersection surface (both irreducible) then a general hyperplane section of the cubic surface gives a twisted cubic curve $\Delta_1$, and a general hyperplane section of the quartic surface gives a complete intersection quartic curve $\Delta_2$ in $\PP^3$, and $\Delta_1$ and $\Delta_2$  meet in  at least $d$ points in $\PP^3$. Now $\Delta_2$ lies in a pencil of quadrics, and by irreducibility of $\Delta_2$ at most one of these quadrics contains $\Delta_1$. So a general such quadric meets $\Delta_1$ in 6 points, hence the number of intersection points of $\Delta_1$ and $\Delta_2$ is at most 6. In our situation we have that $d$ is odd and $d \geq 6$, so we have a contradiction.
\end{proof}

\section{One of the curves is ACM}\label{sec: ACM curves}

Although we do not prove a global bound for all curves $C_1,C_2$,  in this section we show that $B(d_1,d_2)$ is a global bound when one of the curves is ACM, i.e.,  the first part of Conjecture \ref{conjecture b} is true at least in this setting. The approach here was inspired by the paper of Diaz \cite{diaz}. Since Hartshorne and Mir\'o-Roig established bounds on the number of intersection points in $\PP^3$ for ACM curves \cite{hartshorne2015}, it is worth exploring a generalization of this situation in $\PP^4$. (We note that in \cite{hartshorne2015} the authors did much more than finding a bound given $d_1$ and $d_2$ for ACM curves; they gave their bounds for given $h$-vectors, although they were forced to make other assumptions. Our situation is simpler, but we are in $\PP^4$ rather than $\PP^3$.) Since we have a global bound among curves on a cubic surface (Theorem \ref{bd on cubic}), and Corollary \ref{cor of C int S} shows that  if only one lies on a cubic surface (say $C_1$) then $| C_1 \cap C_2| \leq 2d_2-1$, we will assume here that neither curve lies on a cubic surface.

Note that neither $d_1$ nor $d_2$ can be 4 or 5 by the assumption that neither curve lies on a cubic surface (see Lemma \ref{quintic}).

In the  proof of Theorem \ref{acm curves}  we will use the following standard facts.

\begin{enumerate}
    \item If $C$ is an ACM curve with $h$-vector $(1,3,a_2,\dots,a_s)$ then the regularity of $I_C$ is $s+1$. (For example, if $C$ is a rational normal curve in $\PP^4$ then its $h$-vector is $(1,3)$ and the regularity of $I_C$ is 2.)

    \item If $C$ is any curve and the regularity of $I_C$ is $r$ then $I_C$ is generated in degree $\leq r$ (a result of Castelnuovo -- see \cite[Lecture 14]{mumford}).
\end{enumerate}

Note that if $C$ is not ACM then knowing the regularity of a general hyperplane section of $C$ does not help to find the regularity of $I_C$. This is the reason that we need to assume that one of the curves is ACM in Theorem \ref{acm curves}.
Before coming to that we need a simple general statement on the regularity of ACM curves.
\begin{lemma}\label{lem: upper bound on reg}
Let $C$ be a reduced, irreducible, nondegenerate ACM curve in $\PP^4$ of degree $d$ not contained in a cubic surface. Then   
$$ \displaystyle \reg(I_{C}) \leq \left  \lfloor \frac{d}{4} \right \rfloor +2.$$
\end{lemma}
\begin{proof}
By Lemma \ref{quintic} it must be $d\geq 6$. If $d \in \{ 6,7,8,9 \}$ then the $h$-vector of $C$ must be one of the following
\[
\begin{array}{l}
(1,3,2), \\ 
(1,3,3), \\ 
(1,3,4), \ (1,3,3,1), \\ 
(1,3,3,2), \ (1,3,4,1), (1,3,5)
\end{array}
\]
so the assertion holds.
For $d \geq 10$ and since $C$ is ACM, in order for its $h$-vector to end as late as possible, it must have the form
\[
(1,3,4,\dots,4,a,b)
\]
where $a,b$ are described in Remark \ref{a,b}. (That remark was interested in $C_1 \cup C_2$, but the same argument applies to just to one of the curves.) Then this is a simple computation based on the $h$-vector form to conclude the claim.
\end{proof}

\begin{theorem} \label{acm curves}
    Let $C_1,C_2$ be reduced, irreducible, nondegenerate  curves in $\PP^4$, neither lying on a cubic surface. Assume that one of them is ACM.  Then $| C_1 \cap C_2| \leq B(d_1,d_2)$ (i.e., the first part of Conjecture \ref{conjecture b} holds). 
\end{theorem}

\begin{proof}

Without loss of generality, assume that $C_2$ is ACM.

We first note that if $d_1$ and $d_2$ are both $\leq 9$ then Theorem \ref{B minus Bg} gives $|C_1 \cap C_2| < B(d_1,d_2)$ except for the cases $(d_1,d_2) = (7,7), (7,9)$, $(9,7)$ and $(9,9)$. For $(7,7)$, we can have $B_g = B$ but we need both $C_1$ and $C_2$ to have arithmetic genus 0, which is impossible for ACM curves of degree 7 by Riemann-Roch. The same is true for $(9,9)$. 
When $(d_1,d_2) = (7,9)$ or (9,7), we have $B_g = 26$ and $B = 25$. Achieving $|C_1 \cap C_2| = 26 = B_g$ would require both curves to have arithmetic genus 0. Achieving $|C_1 \cap C_2| = 25$ would require one to have arithmetic genus 0 and the other to have arithmetic genus 1.  In either case we again  use the fact that ACM curves in $\PP^4$ of degree 7 or 9 cannot have arithmetic genus 0. Thus $|C_1 \cap C_2| < B(7,9)$.

From now on we assume that
$C_2$ is ACM, and $ \max \{d_1, d_2 \} \geq~10$.
\medskip

We will prove the assertion first in a couple of special cases and then provide a general argument, which covers all other cases. Let $H$ be a general hyperplane and $Z_2 = C_2 \cap H$.
\medskip

\noindent
\textbf{Case 1.} We assume that $d_2=6$. 

The $h$-vector $(1,3,2)$ plus the fact that $Z_2 $ has the UPP (since $C_2$ is irreducible), force such a set of $6$ points to have an ideal generated by quadrics (as one can verify by a case-by-case argument, using UPP). Hence there exists a quadric hypersurface vanishing on $C_2$ and not $C_1$, so since $d_1 \geq 10$ we have 
$$|C_1 \cap C_2| \leq 2d_1 < 3(d_1-1) = B(d_1,6).$$
\textbf{Case 2.} We assume that $d_2=7$ and that the $h$-vector of $C_2$ is $(1,3,3)$. 

Thus $C_2$ (being ACM) lies in a 3-dimensional vector space of quadrics. If the base locus of this linear system is 1-dimensional then it defines a complete intersection linking $C_2$ to a line. This line is not a component of $C_1$, so there is a quadric hypersurface containing $C_2$ but not $C_1$, and as before 
\[
< 3d_1-2 \leq B(d_1,7)
\]

Now we show that the base locus cannot be of dimension $2$. So we assume to the contrary that the base locus is $2$-dimensional.  The base locus of $[I_{Z_2 |H}]_2$ consists of the seven points of $Z_2$ plus some curve, $Y$, possibly containing some of the points of $Z_2$. The $h$-vector $(1,3,3)$ restricts the possibilities for $Y$ to the following short list:
\begin{itemize}
    \item $Y$ is a line containing $3$ of the points of $Z_2$;
    \item $Y$ is a plane conic containing $5$ of the points of $Z_2$;
    \item $Y$ is a twisted cubic containing all the points of $Z_2$.
\end{itemize}
The first two are impossible because $Z_2$ has LGP (in particular). By the ACM property of $C_2$, in the third case,  $C_2$ must lie on a cubic surface, a contradiction.
\medskip

\noindent
\textbf{Case 3.} We assume that $d_2 = 8$ and that the $h$-vector is $(1,3,3,1)$.

Since $Z_2$ has the UPP as above,  in particular it has the Cayley-Bacharach property. Then by the main result of \cite{DGO}, $Z_2$ is arithmetically Gorenstein, and hence so is $C_2$ since it is ACM.
An arithmetically Gorenstein set of $8$ points in $\PP^3$ is either a complete intersection or else it lies on a smooth twisted cubic curve 
(by \cite[Proposition 9.6 and Remark 9.8]{EP} plus UPP) so $C_2$ is either a complete intersection or else it lies on a cubic surface. We have assumed that the latter is not the case, so $C_2$ is a complete intersection. Thus $I_{C_2}$ is generated in degree 2, so there exists a quadric hypersurface $F$ vanishing on $C_2$ and not containing $C_1$. This means 
$$|C_1 \cap C_2| \leq 2d_1 < B(d_1,8).$$

\noindent
\textbf{Case 4.} We assume that $d_2 = 9$ and the $h$-vector is either $(1,3,4,1)$ or $(1,3,3,2)$.

We deal first with the $h$-vector of $C_2$ being $(1,3,3,2)$. Let $X$ be $Z_2=C_2\cap H$ with one point removed. The $h$-vector of $X$ must be $(1,3,3,1)$. Indeed, the $h$-vector $(1,3,2,2)$ is excluded by Remark \ref{rem:lowdeg}. So there is a $3$-dimensional linear system of quadrics containing $X$ and the same linear system contains all $9$ points in $Z_2$. In particular $X$ cannot be a complete intersection $(2,2,2)$ because then the remaining point would impose a new condition on quadrics. Consequently $X$ lies on a twisted cubic and so does $Z_2$. This implies that $C_2$ lies on a cubic surface, which contradicts assumptions of the Theorem.

Now we assume that the $h$-vector of $C_2$ is $(1,3,4,1)$. So we have a set of 9 points with UPP in $H = \PP^3$ with this $h$-vector. These points lie on a pencil of quadric surfaces in $H$, and by UPP the intersection curve of these surfaces is an irreducible quartic curve $Y$. The $h$-vector forces the ideal of $C_2 \cap H$ to have at least 3 minimal generators of degree 3. We claim that there is no quartic generator. Suppose there were one, and consider the artinian ideal $J$ in $S = \CC[x,y,z]$ generated by the images of the two quadric and three cubic generators. The Hilbert function of $S/J$ would be $(1,3,4,1,1,\dots)$. This means the base locus of the corresponding ideal $I_{C_2 \cap H} \subset \CC[w,x,y,z]$ contains a line which contains at least 4 of the points, violating UPP. It follows that there is a cubic hypersurface in $\PP^4$ containing $C_2$ but not $C_1$, so 
$$
|C_1 \cap C_2| \leq 3d_1 <4(d_1-1)+1  \leq B(d_1,9).
$$

\noindent \textbf{Case 5.} In this case we assume that we are not in any of Cases 1-4. Our argument now involves the regularity $\reg(I_{C_2})$.
We assume that the Theorem is false, i.e.,
$$|C_1 \cap C_2| > B(d_1,d_2)$$
and seek  a contradiction.
Let
\[
A = \left  \lfloor \frac{B(d_1,d_2)}{d_1} \right \rfloor.
\]
\noindent \underline{Claim 1}: $\reg(I_{C_2}) > A$.

\vspace{.1in}

{\color{black} (The following argument for Claim 1 actually holds also in Cases 1--4.)} Let $F \in [I_{C_2}]_A$ be an arbitrary element. Suppose that $F$ does not also vanish on all of $C_1$. Since $F$ vanishes at all the intersection points, we have
\[
|C_1 \cap F | \geq |C_1 \cap C_2| > B(d_1,d_2).
\]
On the other hand, since $F$ does not vanish on all of $C_1$, we have 
\[
|C_1 \cap F| \leq d_1 A \leq B(d_1,d_2)
\]
by B\'ezout's theorem. 
This contradiction proves that  in our situation, $F$ must also  vanish on all of $C_1$. But in degrees $t \geq \reg(I_{C_2})$, the base locus of $[I_{C_2}]_t$ is exactly $C_2$.  Thus $A = \deg(F) < \reg (I_{C_2})$. This proves Claim 1.

\vspace{.2in}

\noindent \underline{Claim 2}: 
$\reg (I_{C_2}) \leq A$.  

{\color{black} One can check by inspection that when $d_2 = 8$ with $h$-vector $(1,3,4)$ and when $d_2 = 9$ with $h$-vector $(1,3,5)$ then Claim 2 is true. So taking into account Cases 1--4, we can assume without loss of generality that $d_2 \geq 10$.
}

We consider four subcases, EE, EO, OE, OO, depending on the parity of $d_1$ and $d_2$. 

\vspace{.1in}

\begin{itemize}
\item[\textbf{EE:}] Assume $d_1, d_2$ both even. 

\vspace{.1in}

\begin{itemize}
    \item Consider $d_1 \leq d_2$. By our assumption we know $d_2 \geq 10$, and by (\ref{eq: our bound}) we have
\[
A = \left \lfloor \frac{B(d_1,d_2)}{d_1} \right \rfloor = \left \lfloor \frac{d_2-1}{2} \right \rfloor \geq \left \lfloor \frac{d_2}{4} \right \rfloor +2 \geq \reg(I_{C_2}),
\]
the latter by Lemma \ref{lem: upper bound on reg}. 

\vspace{.1in}

\item Consider $d_1 \geq d_2$.
 Then 
\[
A = \left \lfloor \frac{B(d_1,d_2)}{d_1} \right \rfloor = \left \lfloor \frac{(d_1-1)d_2}{2d_1} \right \rfloor = \left \lfloor \frac{d_1-1}{d_1} \cdot \frac{d_2}{2} \right \rfloor \geq \left \lfloor \frac{d_2-1}{d_2} \cdot \frac{d_2}{2} \right \rfloor = \left \lfloor \frac{d_2-1}{2} \right \rfloor. 
\]

\noindent Since $d_2 \geq 10$ we get by Lemma \ref{lem: upper bound on reg}
\[
A \geq \left \lfloor \frac{d_2}{4} \right \rfloor +2 \geq \reg(I_{C_2}).
\]

\end{itemize}\vspace{.1in}

\item[\textbf{EO:}]  Assume    $d_1$ even and $d_2$ odd.

\vspace{.1in}

\noindent Since  $d_2 > 10$ we  have 
\[
A =     \left \lfloor \frac{B(d_1,d_2)}{d_1} \right \rfloor  = \frac{d_2-1}{2} \geq \left \lfloor \frac{d_2}{4} \right \rfloor +2 \geq \reg(I_{C_2}).
\]

\vspace{.1in}

\item[\textbf{OE:}] Assume  $d_1$ odd, $d_2$ even. Exactly as in the case $d_1,d_2$ both even and $d_1 \geq d_2$ we have
\[
A \geq \left \lfloor \frac{d_2-1}{2} \right \rfloor.
\]
For all even $d_2 \geq 10$ we get $A \geq \reg(I_{C_2})$ as desired. 

\vspace{.1in}

\item[\textbf{OO:}] Assume $d_1, d_2$ both odd. 
$$
A = \left \lfloor \frac{B(d_1,d_2)}{d_1} \right \rfloor  =   \left \lfloor \frac{(d_1-1)(d_2-1)}{2 d_1} + \frac{1}{d_1} \right \rfloor \geq \left \lfloor  \left ( \frac{d_1-1}{d_1} \right ) \left ( \frac{d_2-1}{2} \right ) \right \rfloor \geq \left \lfloor \frac{6}{7} \cdot \frac{d_2-1}{2} \right \rfloor 
$$
since $d_1$ is odd and $\geq 7$. 
Since $d_2 \geq 10$ one checks that
\[
\left \lfloor \frac{6}{7} \cdot \frac{d_2-1}{2} \right \rfloor  \geq \left \lfloor \frac{d_2}{4} \right  \rfloor +2
\]
and we conclude using Lemma \ref{lem: upper bound on reg}.  
\end{itemize}

The contradiction between Claim 1 and Claim 2 completes the proof.
\end{proof}

\begin{remark}
The reason that four distinguished cases appear in the proof above is that Claim 2 is not true in these cases. For example, If $d_2=6$, then regardless of whether $d_1$ is even or odd, we get $A = \left \lfloor \frac{d_1-1}{d_1} \cdot 3 \right \rfloor = 2$ but $\reg(I_{C_2}) = 3$, so the regularity assertion of Claim~2 fails. Similar direct computations can be made in the remaining cases.
\end{remark}

The arguments used in the proof of Theorem \ref{acm curves} work for certain non-ACM curves as well. Let $C$ be a reduced, irreducible, nondegenerate curve in $\PP^4$ not lying on a cubic surface.  Recall that 
    \[
    M(C) = \bigoplus_{t \in \ZZ} H^1(\PP^4, \mathcal I_C(t))
    \]
    is the  Hartshorne-Rao module of $C$. 
   By \cite{hartshorne} Theorem 5.2, $[M(C)]_t = 0$ for $t \gg 0$, so if $C$ is not ACM then there is a last nonzero component.

    \begin{corollary}
        Let $L$ be a general linear form, defining a hyperplane $H$. Let $C_1$ be a reduced, irreducible  curve of degree $d_1$ in $\PP^4$. Let $C_2 \subset \PP^4$ be a non-ACM reduced, irreducible curve of degree $d_2$.  Assume:
        
        \begin{itemize}
            \item[(a)] neither $C_1$ nor $C_2$ lies on a cubic surface,
            \item[(b)] $d_1 \leq d_2$,
            \item[(c)] $M(C_2)$ has the property that  $\times L$ from the penultimate  component of $M(C_2)$ (possibly zero) to the last (nonzero) component is not surjective.
        \end{itemize} 
        Then $|C_1 \cap C_2| < B(d_1,d_2)$. In particular, this is true for any arithmetically Buchsbaum, non-ACM curve. 
    \end{corollary}

    \begin{proof}

    The last sentence comes immediately from (c) and the fact that by definition, an arithmetically Buchsbaum curve is one for which $\times L \colon [M(C)]_t \rightarrow [M(C)]_{t+1}$ is the zero map for all $t$.
        
   Note that $L$ is a non-zerodivisor on $R/I_{C_2}$. Let $Z_2 = C_2 \cap H$ and by abuse of notation denote by $I_{Z_2}$ the ideal $I_{Z_2|H}$ of $Z_2$ in $H$ (thought of as an ideal in $\CC[x_1, x_2, x_3]$). 
   For any integer $t \in \ZZ$ consider the long exact sequence
    \[
    \begin{array}{c}
    0 \rightarrow [I_{C_2}]_{t-1} \stackrel{\times L}{\longrightarrow} [I_{C_2}]_t \rightarrow [I_{Z_2}]_t \rightarrow [M({C_2})]_{t-1} \stackrel{\times L}{\longrightarrow} [M({C_2})]_t \rightarrow \\
    H^1(\mathcal I_{Z_2}(t)) \rightarrow H^2(\mathcal I_{C_2}(t-1))  \rightarrow H^2(\mathcal I_{C_2}(t))
    \end{array}
    \]
    and let $D$ be the cokernel of $\times L \colon [M({C_2})]_{t-1} \rightarrow [M({C_2})]_t$.

    Let $e = \min \{ t \ | \ h^1(\mathcal I_{Z_2}(t)) = 0 \}$. Notice that 
    
    \begin{enumerate}
    \item For any $t \geq e$, $H^1(\mathcal I_{Z_2}(t)) = 0$.
    
        \item $\reg(I_{Z_2}) = e+1$.

 \item \label{equal zero} $H^2(\mathcal I_{C_2}(t)) = 0$ for $t \gg 0$.

 \item \label{surjection} Since $H^2(\mathcal I_{Z_2}(t)) = 0$ for all $t$ (since $Z_2$ is zero-dimensional), we have a surjection \[\times L \colon H^2(\mathcal I_{C_2}(t-1)) \rightarrow H^2(\mathcal I_{C_2}(t))\ \text{for all}\ t. \]

  \item $H^2(\mathcal I_{C_2}(t-1)) \rightarrow H^2(\mathcal I_{C_2}(t))$ is injective for $t \geq e$, hence (using \ref{equal zero}. and \ref{surjection}.) 
  \[
  H^2(\mathcal I_{C_2}(t-1))~=~0 \ \text{for}\ t \geq e.
  \]
        
 \item  Since $H^1(\mathcal I_{Z_2}(t)) = 0$ for $t \geq e$, we have
        $[D]_t = 0$ for $t \geq e$.
    \end{enumerate}

  But item 6. implies  $\times L \colon [M({C_2})]_{t-1} \rightarrow [M({C_2})]_t$ is surjective for $t \geq e$. By assumption~(c), this means that the last nonzero component of $M(C)$ comes in degree $< e$.
In particular, $H^1(\mathcal I_C(e)) = 0$.

    Now we have $H^2(\mathcal I_C(e-1)) = H^1(\mathcal I_C(e)) = 0$. This implies that $I_C$ has regularity $\leq e+1$. But we saw that  $I_{Z_2}$ has regularity $e+1$, so $\reg(I_{C_2}) \leq \reg(I_{Z_2})$.

    Now the proof of this corollary is identical to the one for Theorem \ref{acm curves}.
    \end{proof}

\section{Analysis of genera: Establishing Conjecture \ref{conjecture b} in additional cases}
\label{sec: beyond Bg}

In the present section we continue to consider two curves, neither of which lie on a cubic surface in $\PP^4$.   
A careful analysis of possible $h$-vectors led us to define the genus bound $B_g(d_1,d_2)$. In Section \ref{sec: Bg leq B} we demonstrated how it can be applied to verify Conjecture \ref{conjecture b} under the assumption that the difference $|d_1-d_2|$ is small. In this section we drop this assumption, thereby broadening the applicability of the genus bound to establish Conjecture \ref{conjecture b} in more cases.

\begin{theorem} \label{case d neq d'}

Let $C_1, C_2$ be nondegenerate, irreducible curves in $\PP^4$ of degrees $d_1 \neq d_2$. 
Assume that neither $C_1$ nor $C_2$ lies on a cubic surface.  Assume $d_1$ and $d_2$ are both even and $d_1 + d_2 \equiv 0 \ (\hbox{\rm mod } 4).$   Finally, assume that the general hyperplane  section $\Gamma$ of $C_1 \cup C_2$ has $h$-vector
\[
(1,3, \underbrace{4,\dots,4}_{m} , 3, 1).
\] 
  Then $|C_1 \cap C_2| < B(d_1,d_2)$ so Conjecture \ref{conjecture b} holds in these cases. 
\end{theorem}

\begin{proof}
Assume $d_1<d_2$. Then $d_1+d_2\ge 16$, so from the $h$–vector we have $m\ge2$, and by Lemma \ref{lgp rmk} the scheme $\Gamma$ is a complete intersection of type $(2,2,m+2)$. By Lemma \ref{mig fact}, $C$ is a complete intersection linking $C_1$ to $C_2$.

Write $g,g_1,g_2$ for the genera of $C,C_1,C_2$. From Lemma \ref{genus from hvector1},
\[
g=2m^2+6m+5,
\qquad
d_1+d_2=4(m+2).
\]
Let $d_1=2k$. Then $d_2-d_1=4(m+2-k)$. The liaison genus formula \cite[Cor.~5.2.14]{MiglioreBook} gives
\[
g_2-g_1=\frac{1}{2}(m+1)(d_2-d_1)=2(m+1)(m+2-k),
\]
so in particular $g_2\ge2(m+1)(m+2-k)$.

Let $N=|C_1\cap C_2|$. By Proposition \ref{rosa fact},
\[
N\le g-g_2+1 \le 2m^2+6m+5 - 2(m+1)(m+2-k) + 1 = 2mk+2k+1.
\]

On the other hand,
\[
B(d_1,d_2)
=
\frac{d_1(d_2-1)}{2}
=
\frac{(2k)(4m+8-2k-1)}{2}
=
4mk+7k-2k^2.
\]
Hence
\[
B(d_1,d_2)-N
\ge k(2m+5-2k)-1.
\]

Since $d_1+d_2>2d_1$ we have $2m+4>2k$, so $2m+5-2k>1$, and therefore
\[
B(d_1,d_2)-N>0.
\]
Thus $C_1$ and $C_2$ cannot meet in $B(d_1,d_2)$ points.
\end{proof}

The next result has an interesting feature. We will give a situation where the lifting result of Huneke and Ulrich does not apply, but we are able to show that ``enough" lifting can be forced to occur so that we still obtain the desired result.

\begin{proposition}\label{prop: B holds odd degrees}
    Let $C_1, C_2$ be reduced, irreducible, nondegenerate curves in $\PP^4$, neither lying on a cubic surface, and assume that $d_2 - d_1 = 4$ and that $d_1 + d_2$ is congruent to 2 (mod 4). Then $|C_1 \cap C_2| \leq B(d_1,d_2)$ so Conjecture \ref{conjecture b} holds in these cases.
\end{proposition}

\begin{proof}
First notice that
the arithmetic assumptions on $d_1$ and $d_2$ force these numbers to be odd. Then,
since we always assume that both curves have degree $\geq 6$ and $d_2 - d_1 = 4$, we have for free that $d_2 \geq 11$.
  Using this fact, we first consider what Theorem~\ref{B minus Bg} gives us
$$B(d_1,d_2)-B_g(d_1,d_2)=-2.$$

\noindent 
We have seen in \Cref{lem: curves on del Pezzo quartic} that two irreducible curves with arithmetic genus 0 exist such that $|C_1 \cap C_2| = B(d_1,d_2)$, so we only need to show that $|C_1 \cap C_2| \leq B(d_1,d_2)$. Specifically, we only have to exclude that $|C_1 \cap C_2| = B_g(d_1,d_2)$ and $|C_1 \cap C_2| = B_g(d_1,d_2)-1$.

Let $C = C_1 \cup C_2$. By definition of $B_g(d_1,d_2)$, the $h$-vector for $\Gamma = C \cap H$ used to compute 
it is 
\[
(1,3,\underbrace{4,4,\dots,4}_m,2). 
\]
Since we have assumed that neither $C_1$ nor $C_2$ lies on a cubic surface, no other $h$-vector will result in a bound equal to this $B_g(d_1,d_2)$ or $B_g(d_1,d_2)-1$, thanks to Remark \ref{a,b} and \Cref{genus from hvector1}.
So to rule out $|C_1 \cap C_2| = B_g(d_1,d_2)$ or $B_g(d_1,d_2)-1$ we only have to consider the above $h$-vector. 

Notice that there are two elements of $[I_Z]_{m+2}$ (viewed in the coordinate ring of $H = \PP^3$) that are independent of the complete intersection of two quadrics, since the 4 in degree $m+1$ drops to 2 in degree $m+2$.

If $|C_1 \cap C_2| = B_g(d_1,d_2)$, we have seen in \Cref{rem: max bg implies acm} that this forces $C_1$ and $C_2$ to both have arithmetic genus 0, and also it forces $C$ to be ACM.
This latter fact means that both new forms of degree $m+2$ lift to $I_{C}$. Since we know $C$ lies in a complete intersection of two quadrics, $C$ is linked in a complete intersection of type $(2,2,m+2)$ to an ACM curve, $D$, of degree 2. Such a curve necessarily has $h$-vector $(1,1)$, and hence arithmetic genus 0.
This means that $C_1 \cup D$ is linked to $C_2$, so using the genus formula  \cite[Corollary 5.2.14]{MiglioreBook} we obtain
\[
p_a(C_1 \cup D) - 0 = \frac{1}{2} (2+2+(m+2)-5)(d_1+2-d_2) = -(m+1).
\]
But the union of two curves of arithmetic genus 0 at worst has arithmetic genus $-1$ (when the curves are disjoint) so we have a contradiction, and we have excluded $|C_1 \cap C_2| = B_g(d_1,d_2)$.

Suppose $|C_1 \cap C_2| = B_g(d_1,d_2)-1$. From the proof of \Cref{genus formula} we have the following. Let $L$ be a linear form defining our general hyperplane $H$. Then 
\[
\times L \colon M(C)(-1) \rightarrow M(C)
\]
has a kernel, say $A$, that is a submodule of $M(C)$ with total dimension $k$ (as in \Cref{genus formula}). 
Since $p_a(C)$ is one less than the maximal, there are two possibilities:

\begin{enumerate}

\item $C$ is ACM and  one of the curves (say $C_1$) has arithmetic genus 1 rather than 0 (using \Cref{rosa fact}), or

\item Both curves have arithmetic genus 0 but $C$ is not ACM.

\end{enumerate}

The argument for the first case is the same as the one we gave to eliminate $|C_1 \cap C_2| = B_g(d_1,d_2)$, just replacing 0 with 1 for one of the arithmetic genera. So we suppose $C$ is not ACM. 
Since $p_a(C)$ is one less than the maximal, \Cref{genus formula} gives $k = 1$. In particular, $\dim [A]_{m+2} \leq 1$ (note $[A]_{m+2}$ is a subvector space of $[M(C)]_{m+1}$). Consider the long exact sequence
  \[
    0 \rightarrow [I_C]_{m+1} \stackrel{\times L}{\longrightarrow} [I_C]_{m+2} \stackrel{\sigma}{\longrightarrow} [I_{\Gamma|H}]_{m+2} \rightarrow M(C)_{m+1} \stackrel{\times L}{\longrightarrow} M(C)_{m+2} \rightarrow \cdots
    \]
The cokernel of $\sigma$ is precisely $[A]_{m+2}$, which we have seen has dimension $\leq 1$. Since $I_{\Gamma}$ contains two elements of degree $m+2$ that are independent of the quadrics, this means that at least one of them lifts to $C$. Thus again we have a link of type $(2,2,m+2)$, linking $C$ to a curve $D$ of degree 2 (no longer ACM). 

We next compute the arithmetic genus of $D$. Note that the degree of $C$ is $d_1+d_2 = 4m+6$ and \Cref{genus from hvector1} gives us that the genus $g$ obtained from our $h$-vector (before subtracting $k=1$) is
\[
g = 2(m+1) + 4\binom{m+1}{2}.
\]
The genus formula for the residual now gives 
\[
(g-1) - p_a(D) = \frac{1}{2} (2+2+(m+2))-5)(4m+4)
\]
so after a calculation we get $p_a(D) = -1$.
Arguing as before, we have a link of $C_1 \cup D$ to $C_2$ and $p_a(C_1 \cup D) = -(m+1)$
But the union of a curve of arithmetic genus 0 and one of arithmetic genus $-1$ has arithmetic genus at worst $-2$, so again we have a contradiction.
\end{proof}

\section{Open questions} \label{sec: open questions}

Several questions naturally arise as a result of this work. Among them:

\begin{enumerate}

\item Can we prove Conjecture \ref{conjecture b}? This would give a global upper bound for the number of intersection points and completely describe when the bound is achieved. 

\item Can we give a bound when the given information consists of pairs $(d_1,g_1), (d_2,g_2)$ (where $g_i$ are the arithmetic genera)? We have seen that this has an important effect on the bound.
    
\item What happens in $\PP^n$? Can we find a sharp upper bound on the number of intersection points among curves lying on a surface of minimal degree? Can we describe the curves in $\PP^n$ that achieve the bound, whether or not they lie on such a surface? In particular, must they lie on a surface of minimal degree in order to achieve it?

\item Back to $\PP^3$, what if $C_1$ and $C_2$ do not lie on a surface of given degree $e$? Can we describe sharp bounds in terms of $d_1, d_2$ and $e$?

\item 

Can the approaches developed in Section \ref{sec: beyond Bg} be applied to tackle more cases?

\item In any projective space, can we restrict to special kinds of curves, as was done in Theorem \ref{acm curves} where one of the curves is ACM in $\PP^4$ and in \cite{hartshorne2015} for ACM curves in $\PP^3$, and produce interesting bounds?

\item These ideas can be extended beyond curves. For example, in $\PP^4$, in how many points can a surface $S$ of degree $d_1$ meet a curve $C$ of degree $d_2$? One such  result is Lemma \ref{C int S}. It was shown in \cite{POLITUS4} that they can meet in $d_1 d_2$ points, but (as is the case for curves in $\PP^3$) apart from trivial cases, $C$ and $S$ must be reducible (and a partial description of this situation was given). The paper \cite{POLITUS4} focused primarily on sets of $d_1d_2$ points in $\PP^4$ for which the general projection to $\PP^3$ is the full intersection of a surface of degree $d_1$ and a curve of degree $d_2$, but one can also ask for an upper bound for the number of intersection points in $\PP^4$ of an irreducible surface $S$ of degree $d_1$ and an irreducible curve $C$ of degree $d_2$.  Assume $C$ and $S$ are irreducible, meeting in $N$ points. Project from a general point of $C$. The image is a curve $C'$ of degree $d_2 -1$ and a surface of degree $d_1$. So $N \leq d_1 (d_2-1)$. Can this be improved? 

\end{enumerate}

\vspace{.3in}

\noindent {\bf Declaration of generative AI and AI-assisted technologies in the manuscript preparation process}:
During the preparation of this work the authors used ChatGPT in order to draft initial MATLAB scripts for the creation of the figures in the introduction; these initial scripts were subsequently reviewed and modified by the authors.
\medskip

\noindent {\bf Acknowledgements:}
Chiantini and Favacchio are members of the Italian GNSAGA-INDAM. Chiantini received partial financial support by the Italian PRIN 2022 funds. Favacchio was supported by the funding PREMIO\_SINGOLI\_RIC\_[2025] from the Department of Engineering, University of Palermo.
Migliore was partially supported by Simons Foundation grant \#839618.

\bibliographystyle{siam}

\end{document}